\documentclass[english]{scrartcl}

\usepackage[utf8]{inputenx}
\usepackage[T1]{fontenc}
\usepackage{csquotes,babel}
\usepackage{lmodern}
\usepackage[final]{microtype}
\usepackage{amsmath}
\usepackage{amsfonts}            
\usepackage{amssymb}             
\usepackage{amsthm}              
\usepackage{mathtools}
\usepackage{xcolor}
\usepackage{enumitem}
\usepackage[pdftex,colorlinks=true,urlcolor=red,citecolor=blue,linkcolor=blue]{hyperref}
\usepackage[nameinlink,capitalize]{cleveref}

\numberwithin{equation}{section} 

\newtheorem{theorem}{Theorem}[section]
\newtheorem{lemma}[theorem]{Lemma}
\newtheorem{corollary}[theorem]{Corollary}

\newtheorem{assumption}[theorem]{Assumption}
\newtheorem{remark}[theorem]{Remark}
\newtheorem{definition}[theorem]{Definition}

\newcommand\doi[1]{\href{http://dx.doi.org/#1}{doi: \nolinkurl{#1}}}

\renewcommand\AA{\mathcal{A}}
\newcommand\KK{\mathcal{K}}
\newcommand\N{\mathbb{N}}
\newcommand\R{\mathbb{R}}
\newcommand\MM{\mathcal{M}}
\newcommand\RR{\mathcal{R}}
\newcommand\TT{\mathcal{T}}
\newcommand\NN{\mathcal{N}}

\renewcommand\d{\mathrm{d}}
\newcommand\dist{\operatorname{dist}}
\newcommand\cl{\operatorname{cl}}
\newcommand\weakly{\rightharpoonup}
\newcommand\weaklystar{\stackrel\star\rightharpoonup}
\newcommand\Uad{U_{\mathrm{ad}}}
\newcommand\Ueff{U_{\mathrm{eff}}}
\newcommand\Ustate{U_{\mathrm{state}}}
\newcommand\qsupp{\operatorname{q-supp}}
\newcommand\supp{\operatorname{supp}}
\newcommand\polar{^\circ}
\newcommand\anni{^\perp}
\newcommand\adjoint{^\star}
\newcommand\dualspace{^\star}
\newcommand\esssup{\operatorname{ess\,sup}}

\DeclareMathAlphabet{\mathpzc}{OT1}{pzc}{m}{it}

\DeclarePairedDelimiter\abs{\lvert}{\rvert}
\DeclarePairedDelimiter\norm{\lVert}{\rVert}
\DeclarePairedDelimiterX\dual[2]{\langle}{\rangle}{#1,#2}
\DeclarePairedDelimiterX\innerprod[2](){#1,#2}

\providecommand\given{\nonscript\;\delimsize|\nonscript\;}
\DeclarePairedDelimiterX\set[1]\{\}{#1}

\DeclarePairedDelimiter\parens()
\DeclarePairedDelimiter\bracks[]
\DeclarePairedDelimiter\seq()
\DeclarePairedDelimiter\braces\{\}

\newcommand{\embeds}{\hookrightarrow}

\newenvironment{msc}%
{\par\noindent{\sectfont MSC (2020):}\ }{\par}

\newcommand{\mscLink}[1]{\href{https://www.ams.org/mathscinet/msc/msc2020.html?t=#1}{#1}}

\newenvironment{keywords}%
{\par\noindent{\sectfont Keywords:}\ }{\par}

\usepackage[
	style=authoryear-comp,
	sorting=nyt,
	url=true,
	doi=true,                
	eprint=true,
	maxbibnames=99,          
	maxcitenames=2,
	uniquelist=minyear,
	dashed=false,            
	sortcites=false,          
	backend=biber,           
	safeinputenc,            
]{biblatex}

\DeclareCiteCommand{\cite}{%
	\ifbibmacroundef{cite:init}{}{\usebibmacro{cite:init}}\usebibmacro{prenote}%
}{%
	\usebibmacro{citeindex}%
	\printtext[bibhyperref]{\usebibmacro{cite}}%
}{%
	\ifbibmacroundef{cite:init}{\multicitedelim}{}%
}{%
	\usebibmacro{postnote}%
}%
\DeclareCiteCommand{\parencite}[\mkbibbrackets]{%
	\ifbibmacroundef{cite:init}{}{\usebibmacro{cite:init}}\usebibmacro{prenote}%
}{%
	\usebibmacro{citeindex}%
	\printtext[bibhyperref]{\usebibmacro{cite}}%
}{%
	\ifbibmacroundef{cite:init}{\multicitedelim}{}%
}{%
	\usebibmacro{postnote}%
}%

\let\cite\parencite

\DeclareNameAlias{sortname}{family-given}
\DeclareNameAlias{default}{family-given}

\begin{document}
\title{First-order conditions for the optimal control of the obstacle problem with state constraints\footnote{%
		This research was 
		supported by the German Research Foundation (DFG) within the priority 
		program ``Non-smooth and Complementarity-based Distributed Parameter Systems: 
		Simulation and Hierarchical Optimization'' 
		(SPP 1962) under grant numbers NE 1941/1-2 and WA 3636/4-2.
}}
\author{Ira Neitzel\footnote{%
		Rheinische Friedrich Wilhelms Universität Bonn,
		Institute for Numerical Simulation,
		Friedrich-Hirzebruch-Allee 7, 53115 Bonn, Germany, \url{https://ins.uni-bonn.de/staff/neitzel}
}
\and
Gerd Wachsmuth\footnote{%
	Brandenburgische Technische Universität Cottbus-Senftenberg, 
	Institute of Mathematics, 
	03046 Cott\-bus, Germany, 
	\url{https://www.b-tu.de/fg-optimale-steuerung}
}}
\maketitle
\begin{abstract}
	We consider an optimal control problem
	in which the state is governed by an unilateral obstacle problem
	(with obstacle from below) and
	restricted by a pointwise state constraint (from above).
	In the presence of control constraints,
	we prove, via regularization of the state constraints, that a system of C-stationarity
	is necessary for optimality.
	In the absence of control constraints,
	we show that local minimizers are even strongly stationary
	by a careful discussion of the primal first-order conditions
	of B-stationary type.
\end{abstract}
\begin{msc}
	\mscLink{49K21},
	\mscLink{35J86}
\end{msc}
\section{Introduction}

In this paper, we analyse an optimal control problem subject to an obstacle (from below) and subject to an additional pointwise state constraint (from above).
More precisely, we are interested in the problem
\begin{equation}
	\label{eq:P:unreg}
	\tag{$\mathbf{P}$}	
	\begin{aligned}
		\text{Minimize}&\quad J(y,u) \\
		\text{with respect to}&\quad (y,u)\in H_0^1(\Omega)\times L^2(\Omega)\\
		\text{such that}& \quad y\in K,\; \dual{\AA y-u}{v-y} \ge 0\quad\forall v\in K, \\
		& \quad u \in \Uad \\
		\quad \text{and}& \quad y\le y_b\quad\text{a.e.~in } \Omega.
	\end{aligned}
\end{equation}
In our notation, $y$ is the state and $u$ is the control, and
the set $K$ is defined by the lower obstacle $y_a$,
i.e.,
\begin{equation*}
	K := \set{ v \in H_0^1(\Omega) \given v \ge y_a \text{ a.e.\ in }\Omega }.
\end{equation*}
For the precise assumptions we refer to \cref{asm:big_list}.
We just mention that we work with minimal regularity,
i.e., the coefficients in the differential operator $\AA$
are just assumed to be measurable and bounded,
and we only require
the so-called
``uniform exterior cone condition''
of the bounded domain $\Omega \subset \R^n$, $n \in \set{2,3}$.

Since the solution operator $S$ of the obstacle
is not differentiable,
its optimal control is a challenging problem.
The first contribution is the seminal work
\cite{Mignot1976},
which demonstrates the directional differentiability of $S$
and
in which optimality conditions of strongly stationary type were derived.
Another classical work is \cite{Barbu1984},
in which regularization methods are used to provide stationarity conditions.
It turns out that regularization techniques are applicable to a wider range of problems,
in particular to problems including control constraints.
However, the resulting optimality conditions are weaker than strong stationarity.
It seems that the so-called system of C-stationarity is the best system
which can be derived in this way.
For some more recent contributions to the optimal control of the obstacle problem,
we refer to
\cite{
	HintermuellerSurowiec2011,
	KunischWachsmuth2011,
	SchielaWachsmuth2013,
	Wachsmuth2013:2,
	Wachsmuth2014:2,
	HarderWachsmuth2017:2%
}.

PDE-constrained optimal control problems with pointwise state constraints have also been known as a challenging problem class with respect to optimality conditions for quite some time. The theory is typically based
on the so called Slater condition, for which continuity of the states is usually required. Lagrange multipliers in the first-order optimality system are then
obtained in the space of regular Borel measures, see \cite{cas86}, which in turn leads to low
regularity of the adjoint state. From the meanwhile very rich literature on different aspects of purely state-constrained problems we refer only to \cite{cas93}, where the results of \cite{cas86} have also been extended to boundary control of semilinear elliptic PDEs, to \cite{rayzid98,rayzid99} for some earlier results on (semilinear) parabolic problems, as well as to related problems with constraints of bottleneck type, \cite{bertr99}. We also mention \cite{hinkun09} for problems with control, state and gradient constraints. The structure of Lagrange multipliers has for instance been discussed in \cite{berkun02a}. More recently, in \cite{CasasMateosVexler2014}, the authors showed improved regularity for the Lagrange multiplier, inspired by a result for sparse optimal controls, see \cite{PieperVexler:2013}. Sparse control became a field of active research rather recently.
These problems resemble state-constrained problems in the way that if for instance measures as control variables are considered, this leads to regularity difficulties in the state equation, instead of the adjoint equation in case of pointwise state constraints. 

The challenges and regularity issues associated with the Lagrange multiplier also influence all further analysis, such as second-order sufficient conditions (SSC) for nonconvex problems, analysis of solution algorithms, or numerical analysis of such problems. For an introductory overview on general aspects of second-order sufficient conditions and finite element error analysis for PDE-constrained optimization, not restricted to pointwise state constraints, we mention \cite{castr15} or \cite{HinzeTroeltzsch:2010}, respectively.

There are a number of well-established regularization techniques that help to avoid theoretical and numerical difficulties associated with pointwise state constraints. We mention the rather classical Moreau-Yosida regularization approach from
\cite{itokun03} that we will pursue in this paper, a Lavrentiev-regularization technique presented originally in \cite{meyrtr06}, as well as the virtual control regularization approach from \cite{kruroe09}.
Moreover, barrier methods are in use; we refer for instance to \cite{sch09} or  \cite{ulbulb09}.
As the literature on regularization of pointwise state constraints is meanwhile also rather rich, let us only refer to \cite{berhadhin00, hinitokun03,hintryou08} as examples of a few publications related to the Moreau-Yosida penalization. 

We also would like to point out \cite{rtr07} and the references therein, where elliptic and parabolic problems with mixed control-state constraints have been considered. Lagrange multipliers are shown to exist in $L^p$-spaces. While these constraints exhibit better regularity properties, the analysis begins with existence of multipliers in the dual space of $L^\infty$, which is even less regular than the space of regular Borel measures. The regularity of the multipliers and also the optimal control is subsequently improved. For a problem with bilateral control and mixed control-state constraints, a separation condition for the active sets allows to prove $L^1$-regularity of the multipliers and further regularity improvements via bootstrapping arguments.
Such a separation condition has also been used in e.g.\ \cite{altgrmetroe10} for stability analysis of linear-quadratic elliptic problems with mixed constraints, for convergence analysis of the SQP method for nonlinear problems in \cite{grimetr08}, and in \cite{neitr09} for Lavrentiev regularization of pointwise state constraints in parabolic problems.

In our analysis, we will rely on separate supports of the multipliers associated with the obstacle and the state constraints. In essence, this condition allows to apply typical cut-off-type arguments. If the support of a low-regularity Lagrange-multiplier is clearly separated from another point or rather area of interest, the smoothing properties of the solution operators can be used to prove higher regularity on appropriate subdomains. In our case, this means that the adjoint state admits $H_0^1-$regularity away from the support of the state-constraint multiplier. Cut-off arguments are a typical strategy to prove known higher interior regularity results for PDEs. For PDE-constrained optimization, such techniques have for example also been used
 to consider Dirichlet boundary control of Poisson's equation with pointwise state constraints in the interior, analyzed in \cite{MateosNeitzel:2016} even though the states only admit $H^{\frac{1}{2}}(\Omega)$ regularity on the whole domain. 

The literature concerning optimal control of the obstacle problem
with additional state constraints is rather scarce.
We are only aware of three publications.

In
\cite{He1987},
the author
considers a problem which is more general than \eqref{eq:P:unreg}.
In \cite[Theorem~6.2]{He1987},
a system of C-stationarity is derived
which includes a multiplier $\nu \in L^2(\Omega)$
for the state constraint.
This unusual high regularity seems to be related
to the requirement (6.5) therein.
It is assumed
that perturbations $z \in L^2(\Omega)$
in the state constraint lead to
perturbations of the optimal value
which can be bounded from below (up to first order) by the $L^2(\Omega)$-norm of $z$.
It is not clear whether this assumption can be verified
for a large class of problems.
We do not expect that the multiplier of the state constraint considered in our paper belongs to $L^2(\Omega)$.

In the contribution
\cite{Bergounioux1998},
a more general variational inequality is considered.
This variational inequality is regularized
and
optimality conditions for the regularized problem
(subject to the state constraints)
are derived.
The passage to the limit in this optimality system is not addressed.

\cite[Section~3]{BergouniouxTiba1998}
addresses a problem very similar to \eqref{eq:P:unreg},
but the state constraint is defined by a closed convex set in $H_0^1(\Omega)$.
Again, the obstacle problem is regularized.
This contribution also addresses
the passage to the limit in the optimality system.
The resulting optimality system
uses a limiting object defined in
\cite[Definition~3.1]{BergouniouxTiba1998}.
It is not clear how much information is carried by this object.
We just mention that optimality systems defined
by the so-called limiting normal cone
seems to be of limited use in infinite dimensions,
see \cite{MehlitzWachsmuth2017:1,HarderWachsmuth2017:1}.

In this work,
we present three optimality systems.
First,
in \cref{thm:B_stat_second}
we derive a primal optimality system
(i.e., it does not involve dual multipliers),
which is typically called B-stationarity.
To this end, we heavily utilize
that the pointwise convexity of the solution operator of the obstacle problem
(see \cref{lem:pw_convex})
renders the state constraint
convex w.r.t. the control,
i.e., the set $\Ustate := \set{u \in L^2(\Omega) \given S(u) \le y_b}$
is convex,
where $S$ is the solution operator of the obstacle problem,
see \cref{lem:mask_state_constraint}.

Second,
we use a regularization approach to derive optimality conditions
of C-stationary type, see \cref{thm:c_stationarity}.
To this end, we just regularize the state constraints
and apply the C-stationarity conditions
from \cite{Wachsmuth2014:2}
to the regularized problems.
The passage to the limit in the optimality system
requires some care
due to the multiplier of the state constraint,
which is a measure and appears on the right-hand side of the adjoint equation.

Finally,
we show that the system of B-stationarity
is equivalent to strong stationarity in the absence of control constraints,
see \cref{thm:strong}.
To this end,
we utilize classical results by \cite{Mignot1976}.
Due to the state constraints,
this is much more difficult than in the classical setting.
We mention that this result also allows
to characterize the normal cone to the convex set $\Ustate$.

The paper is structured as follows.
In \cref{subsec:notation} we fix some notation.
Linear PDEs governed by the differential operator $\AA$
and its adjoint $\AA\adjoint$
are discussed in \cref{subsec:sol_eq}.
Of particular importance are
\cref{lem:hinfboundu,lem:bound_adjoint}.
In these results,
we discuss PDEs with irregular but localized right-hand sides
and show that the solution enjoys higher regularity
if we neglect a neighborhood of the support of the right-hand side.
Afterwards,
the obstacle problem is discussed in \cref{subsec:solution_operator_obstacle}.
Using the regularity results of the previous section,
we can provide uniform estimates for the directional derivative $S'$
away from the active set $\set{y = y_a}$,
see \cref{lem:better_differentiability}.
In \cref{subsec:oc},
we collect all the assumptions,
see \cref{asm:big_list},
and give some basic properties of \eqref{eq:P:unreg}.
The optimality conditions are discussed in
\cref{sec:primal_systems,sec:regularization,sec:strong_stat},
as described above.

\section{Preliminaries and technical results}
\label{sec:prelim}
We start by setting up the notation in \cref{subsec:notation}.
Afterwards,
we discuss the solution mapping of the differential operator in \cref{subsec:sol_eq}.
The variational inequality will be addressed in \cref{subsec:solution_operator_obstacle},
and finally,
we give some basic properties of the optimal control problem in \cref{subsec:oc}.

\subsection{Notation}
\label{subsec:notation}

Let us fix some notation.
The positive numbers are denoted by $\R^+ := (0,\infty)$.
For a (real) Banach space $X$,
we denote by $X\dualspace$ the (topological) dual space of $X$.
The corresponding duality pairing is denoted by
$\dual{\cdot}{\cdot}_X \colon X\dualspace \times X \to \R$.
If the space $X$ is clear from the context, we may omit the index $X$. The inner product in $L^2(\Omega)$
is denoted by
$\innerprod{\cdot}{\cdot}$.

For a subset set $C \subset X$,
we define the
polar cone and the annihilator by
\begin{align*}
	C\polar &:= \set{ \xi \in X^\star \given \forall x \in C: \dual{\xi}{x}_X \le 0}
	, \\
	C\anni &:= \set{ \xi \in X^\star \given \forall x \in C: \dual{\xi}{x}_X = 0}
	,
\end{align*}
respectively.
Analogously, we define $D\polar, D\anni \subset X$ for $D \subset X\dualspace$.
In particular, the annihilator of $\xi \in X^\star$ is defined via
\begin{equation*}
	\xi\anni := \set{x \in X \given \dual{\xi}{x}_X = 0}.
\end{equation*}
Now, let $C \subset X$ be closed and convex.
For all $x \in C$,
we define the
radial cone, the tangent cone, and the normal cone
to $C$ at $x$
via
\begin{align*}
	\RR_C(x) &:=
	\bigcup_{\lambda > 0} (C - x),
	&
	\TT_C(x) &:= \cl\parens*{\RR_C(x)},
	&
	\NN_C(x) &:= \TT_C(x)\polar = (C - x)\polar,
\end{align*}
respectively.
In case $x \in X \setminus C$,
we define all these cones to be the empty set.
In case $X = L^2(\Omega)$,
where $\Omega \subset \R^n$ is measurable,
we identify $X^\star$ with $X$ in the canonical way
and, therefore, interpret $\NN_C(x)$ as a subset of $X$.

Further, we need some basic concepts of
capacity theory.
For a summary, we refer to
\cite[Section~2]{Wachsmuth2013:2},
\cite[Section~1.2]{Wachsmuth2014:2}
and
\cite[Section~3]{HarderWachsmuth2017:2}.
In particular, we
require the definition of
the quasi-support $\qsupp(\xi)$ for measures $\xi \in H^{-1}(\Omega)^-$,
see \cite[Lemma~3.1]{Wachsmuth2013:2}
and \cite[Lemma~3.7, Definition~3.8]{HarderWachsmuth2017:2}
(called ``fine support'' therein).

If $y \colon \Omega \to \R$ is some function,
we define
\begin{equation*}
	\set{ y > 0 }
	:=
	\set{ x \in \Omega \given y(x) > 0 }
	.
\end{equation*}
Note that this set is defined up to sets of measure zero if $y \in L^r(\Omega)$, $r \in [1,\infty]$,
and up to sets of capacity zero if $y \in H^1(\Omega)$.
The same notation will be used for other relations
and also with more than one function, e.g.,
$\set{y_1 > y_2}$
with the obvious meaning.

\subsection{Solution operators of differential equations}
\label{subsec:sol_eq}
In this section, we specify the properties of the differential operator $\AA$ appearing in \eqref{eq:P:unreg}. To this end, let us define $\AA \colon H^1(\Omega) \to H^{-1}(\Omega)$
via
\begin{equation}
	\label{eq:diff_op}
	\begin{aligned}
		\dual{\AA y}{v}
		&:=
		\int_\Omega
		\sum_{i,j = 1}^n a_{ij} \, \partial_j y \, \partial_i v
		+
		\sum_{i = 1}^n b_i \, y \, \partial_i v
		+
		\sum_{i = 1}^n c_i \, \partial_i y \, v
		+
		d \, y \, v
		\, \d x
		\\
		&=
		\int_\Omega
		\nabla v^\top A \, \nabla y
		+
		y \, \parens{b^\top \nabla v}
		+
		v \, \parens{c^\top \nabla y}
		+
		d \, y \, v
		\,
		\d x
		\qquad\forall y \in H^1(\Omega), v \in H_0^1(\Omega).
	\end{aligned}
\end{equation}
For the coefficients appearing in \eqref{eq:diff_op},
we assume measurability and boundedness,
i.e.,
\begin{equation}
	\label{eq:regularity_coefficients}
	a_{ij}, b_i, c_i, d \in L^\infty(\Omega)
\end{equation}
on a domain $\Omega \subset \R^n$, $n \in \set{2,3}$, satisfying a uniform exterior cone condition.
Note that this is precisely the definition of the operator \cite[(8.1)]{GilbargTrudinger2001}
and the regularity condition of $\Omega$ described on page~205 therein.
As usual, we will denote by $\Gamma:=\partial\Omega$ the boundary of $\Omega$.
For our analysis, we assume that $\AA$ is strictly elliptic, i.e.,
there exists $\gamma_0 > 0$ such that
\begin{equation}
	\label{eq:elliptic}
	\sum_{i,j=1}^n a_{ij}(x) \, w_i \, w_j
	=
	w^\top A(x) \, w
	\ge
	\gamma_0 \, \norm{w}_{\R^n}^2
	\qquad
	\forall w \in \R^n \text{ and a.a. } x \in \Omega.
\end{equation}
Further, $\AA$ should be coercive on $H_0^1(\Omega)$,
i.e., there exists $\gamma_1 > 0$
such that
\begin{equation}
	\label{eq:coercivity}
	\dual{\AA y}{y}
	\ge
	\gamma_1 \, \norm{y}_{H_0^1(\Omega)}^2
	\qquad\forall y \in H_0^1(\Omega).
\end{equation}

These assumptions on $\AA$ hold throughout the paper.
In the remainder of this section, we are concerned with existence and regularity results for the differential equation
\begin{equation}\label{eq:Aeq}
	y \in H_0^1(\Omega), 
	\qquad
	\AA y = f
	\quad\text{in } H^{-1}(\Omega)
\end{equation}
and the associated adjoint equation.
The following is a standard existence and regularity result and the starting point of our analysis.
\begin{theorem}\label{thm:existenceh1}
Let $f\in H^{-1}(\Omega)$. Then there exists a unique weak solution $y\in H_0^1(\Omega)$ of \eqref{eq:Aeq}, satisfying
\begin{equation}\label{eq:h1bound}
\norm{y}_{H_0^1(\Omega)}\le C\norm{f}_{H^{-1}(\Omega)}.
\end{equation}
\end{theorem}
\begin{proof}
Under our assumptions on $\AA$,  this follows by the well-known Lax-Milgram theory.
\end{proof}
If we can rely on more regularity of the data, we also obtain better regularity properties of the solution.
However, we are limited by the low regularity \eqref{eq:regularity_coefficients} of the coefficients.
The following theorem,
that particularly includes the case of right-hand sides in $L^2(\Omega)$ for spatial dimensions up to $n=3$,
guarantees Hölder regularity.
In the sequel,
this will allow to rely on a so-called Slater condition as constraint qualification
in order to obtain first-order necessary optimality conditions for controls in $L^2(\Omega)$.

\begin{theorem}\label{thm:hoelderreg}
For every $q' > n$ and
$f\in W^{-1,q'}(\Omega)$, the unique solution $y \in H_0^1(\Omega)$ of equation \eqref{eq:Aeq}
enjoys the additional regularity
$y\in C_0^{0,\alpha}(\Omega)$ for some $\alpha>0$
and we have the estimate
\begin{equation}\label{eq:Aeqh1}
	\norm{y}_{H^1_0(\Omega)}+\norm{y}_{C_0^{0,\alpha}(\Omega)}\le C\norm{f}_{W^{-1,q'}(\Omega)}.
\end{equation}
Here, $C$ and $\alpha$ do not depend on $f$.
\end{theorem}
\begin{proof}
The additional Hölder regularity follows from \cite[Theorem 8.29]{GilbargTrudinger2001}.
Together with \cite[Theorem 8.15]{GilbargTrudinger2001}, the norm estimate follows.
\end{proof}
Using this regularity,
we can define the solution operator $T \colon W^{-1,q'}(\Omega) \to C_0(\Omega)$
of \eqref{eq:Aeq}.
Note that $T$ is compact,
since
the Hölder space $C_0^{0,\alpha}(\Omega)$
is compactly embedded in $C_0(\Omega)$
by the Arzelà-Ascoli theorem.
Its adjoint operator $T\adjoint \colon \MM(\Omega) \to W_0^{1,q}(\Omega)$
is related to the adjoint equation of \eqref{eq:Aeq}.
\begin{theorem}\label{thm:existenceadjoint}
Let $q \in \parens[\big]{1, n/(n-1)}$ be given.
For every  $\nu \in \MM(\Omega)$, the equation
\[
\AA^\star p=\nu
\]
admits a unique very weak solution $p \in W_0^{1,q}(\Omega)$, i.e.
\begin{equation}\label{eq:adjveryweak}
	\dual{\AA z}{p}_{W_0^{1,q}(\Omega)}
	=
	\int_\Omega z \, \d\nu
	\qquad\forall z \in Z,
\end{equation}
where
\begin{equation}
	\label{eq:def_Z}
	Z := \set{
		z \in H_0^1(\Omega)
		\given
		\AA z \in W^{-1,q'}(\Omega)
	}.
\end{equation}
Here, $q' \in (n, \infty)$ is the conjugate exponent of $q$.
This solution fulfills the estimate
\begin{equation}\label{eq:Aeqwq1pm}
\norm{p}_{W_0^{1,q}(\Omega)}\le C\norm{\nu}_{\MM(\Omega)},
\end{equation}
with $C$ independent of $\nu$.
Finally, if $\nu_k \weaklystar \nu$,
we have $p_k \to p$ in $W_0^{1,q}(\Omega)$
for the corresponding solutions.
\end{theorem}
Note that $Z \subset C_0(\Omega)$
due to \cref{thm:hoelderreg},
hence, the right-hand side in \eqref{eq:adjveryweak} is well defined.
\begin{proof} 
For $\nu \in \MM(\Omega)$,
we define the function $p := T\adjoint \nu \in W_0^{1,q}(\Omega)$.
It is uniquely determined by the properties of the adjoint operator,
\begin{equation}\label{eq:adjsstar}
	\dual{f}{p}_{W_0^{1,q}(\Omega)}
	=
	\dual{\nu}{T f}_{\MM(\Omega),C_0(\Omega)}
	=
	\int_\Omega (T f) \, \d\nu
	\qquad
	\forall f \in W^{-1,q'}(\Omega).
\end{equation}
From this, we observe that \eqref{eq:Aeqwq1pm} holds.
Note that \eqref{eq:adjsstar} is equivalent
to the very weak formulation \eqref{eq:adjveryweak}.

In order to verify the compactness property,
we use that $T\adjoint$
maps weak-$\star$ convergent sequences to weak-$\star$ convergent sequences
(since it is an adjoint operator)
and
it maps bounded sequences to sequences possessing a strong accumulation point
(since it is compact by Schauder's theorem).
Hence,
if $\nu_k \weaklystar \nu$ in $\MM(\Omega)$,
every subsequence of $\seq{T\adjoint \nu_k}_{k \in \N}$
possesses a strong accumulation point
which has to coincide with $T\adjoint \nu$
due to $T\adjoint \nu_k \weaklystar T\adjoint \nu$ in $W_0^{1,q}(\Omega)$.
Hence, a subsequence-subsequence argument
shows $T\adjoint \nu_k \to T\adjoint \nu$ in $W_0^{1,q}(\Omega)$.
\end{proof}

We point out that using the density of $C_0(\Omega)$ in $W^{-1,q'}(\Omega)$,
it is also possible to use
\begin{equation*}
	\hat Z := \set{
		z \in H_0^1(\Omega)
		\given
		\AA z \in C_0(\Omega)
	},
\end{equation*}
as it was done in \cite{CasasMateosVexler2014}.
In particular, this shows that the very weak solution
of \eqref{eq:adjveryweak}
does not depend on the choice of the regularity exponent $q$.

Note that
$\AA z \in W^{-1,q'}(\Omega)$ in the definition of $Z$
means that
\begin{equation*}
	\abs[\big]{ \dual{\AA z}{v} }
	\le
	C
	\norm{v}_{W_0^{1,q}(\Omega)}
	\qquad
	\forall v \in H_0^1(\Omega).
\end{equation*}
Therefore, the functional $\AA z \in H^{-1}(\Omega)$
can be extended continuously to a functional from $W^{-1,q'}(\Omega)$.
Note that we cannot use integration by parts on the left-hand side
of \eqref{eq:adjveryweak}, since this would require
$p \in H_0^1(\Omega)$ or $z \in W_0^{1,q'}(\Omega)$.
This, however, may not hold under the low regularity \eqref{eq:regularity_coefficients}.
For a thorough discussion of the interpretation of the adjoint equation
in the case of coefficients with low regularity,
we refer to \cite{MeyerPanizziSchiela2011}.

Of course,
it is also possible to discuss the adjoint equation with
right-hand sides from $H^{-1}(\Omega)$,
i.e.,
\begin{equation*}
	p \in H_0^1(\Omega),
	\qquad
	\AA\adjoint p = \mu
	\quad\text{in } H^{-1}(\Omega).
\end{equation*}
Existence and uniqueness follows from Lax-Milgram.
Due to $Z \subset H_0^1(\Omega) \cap C_0(\Omega)$,
both notions of solutions coincide
if $\mu \in H^{-1}(\Omega) \cap \MM(\Omega)$.
Indeed,
for every $z \in Z$ and $p \in H_0^1(\Omega)$, we have
\begin{equation*}
	\dual{\AA z}{p}_{W_0^{1,q}(\Omega)}
	=
	\dual{\AA z}{p}_{H_0^1(\Omega)}
	\quad\text{and}\quad
	\int_\Omega z \, \d\mu
	=
	\dual{\mu}{z}_{H_0^1(\Omega)}
	.
\end{equation*}
Thus, we can define a
very weak solution $p \in W_0^{1,q}(\Omega)$
of
\begin{equation*}
	\AA\adjoint p = \mu + \nu
\end{equation*}
for $\mu \in H^{-1}(\Omega)$
and $\nu \in \MM(\Omega)$
via
\begin{equation}\label{eq:adjveryweak2}
	\dual{\AA z}{p}_{W_0^{1,q}(\Omega)}
	=
	\dual{\mu}{z}_{H_0^1(\Omega)}
	+
	\int_\Omega z \, \d\nu
	\qquad\forall z \in Z,
\end{equation}
and this solution does not depend
on the precise splitting of $\mu + \nu$
into $\mu \in H^{-1}(\Omega)$
and $\nu \in \MM(\Omega)$.

Let us consider further properties and auxiliary results, starting with
smooth multipliers for the space $Z$.
\begin{lemma}
	\label{lem:smooth_multiplier}
	Let $z \in H_0^1(\Omega)$ be given such that
	$\AA z \in W^{-1,q'}(\Omega)$ for some $q' > n$
	with $q' < \infty$ in case $n = 2$ and $q' \le 6$ in case $n = 3$.
	Then, for all $\psi \in C^\infty(\Omega)$
	we have
	$\psi z \in H_0^1(\Omega)$
	and
	$\AA( \psi z ) \in W^{-1,q'}(\Omega)$.
\end{lemma}
\begin{proof}
	\cref{thm:hoelderreg} implies $z \in L^\infty(\Omega)$.
	Further, we already know that
	the linear functional
	$v \mapsto \dual{\AA z}{\psi v}$
	belongs to $W^{-1,q'}(\Omega)$.
	Now we consider
	\begin{align*}
		&
		\dual{\AA (\psi z)}{v}
		-
		\dual{\AA z}{\psi v}
		\\&\qquad
		=
		\int_\Omega
		\nabla v^\top A \nabla (\psi z)
		+
		(\psi z) \, \parens{b^\top \nabla v}
		+
		v \, \parens{c^\top \nabla (\psi z)}
		+
		d \, (\psi z) v
		\,
		\d x
		\\&\qquad\qquad
		-
		\int_\Omega
		\nabla (\psi v)^\top A \nabla z
		+
		z \, \parens{b^\top \nabla (\psi v)}
		+
		(\psi v) \, \parens{c^\top \nabla z}
		+
		d z \, (\psi v)
		\,
		\d x
		\\&\qquad
		=
		\int_\Omega
		z \nabla v^\top A \nabla \psi - v \nabla\psi^\top A \nabla z
		-
		z v \, (b^\top \nabla \psi)
		+
		v z \, (c^\top \nabla \psi)
		\,\d x
		.
	\end{align*}
	This yields the bound
	\begin{equation*}
		\abs[\big]{
			\dual{\AA (\psi z)}{v}
			-
			\dual{\AA z}{\psi v}
		}
		\le
		C
		\parens[\big]{
			\norm{\nabla v}_{L^1(\Omega)}
			+
			\norm{v}_{L^2(\Omega)}
			+
			\norm{v}_{L^1(\Omega)}
			+
			\norm{v}_{L^1(\Omega)}
		}
		.
	\end{equation*}
	The Sobolev embedding theorem yields
	$W_0^{1,s}(\Omega) \embeds L^2(\Omega)$
	for $1/s = 1/2 + 1/n$, i.e., $s = 2 n/(n + 2)$.
	Thus, $s = 6/5$ for $n = 3$.
	Due to the assumption on $q'$,
	we have $q' \le s'$, and
	therefore
	$\AA (\psi z) \in W^{-1,q'}(\Omega)$.
\end{proof}

\begin{lemma}\label{lem:hinfboundu}
	Let subsets $U,V\subset\Omega$ be given,
	such that $U$ is compact, $V$ is open
	and $U \subset V$.
	Let $f\in H^{-1}(\Omega)$ and $\varphi\in C^\infty(\mathbb{R}^n)$
	with $\varphi_{|V}=0$. Then the weak solution $y\in H_0^1(\Omega)$ of 
\[
\AA y=\varphi f
\] 
is continuous in a neighborhood of $U$ and fulfills
\begin{equation}
\norm{y}_{C(U)}\le C\norm{f}_{H^{-1}(\Omega) }.
\end{equation}
Moreover, the mapping $H^{-1}(\Omega) \ni f \mapsto y \in C(U)$ is compact.
\end{lemma}
\begin{proof}
 Choose $\psi\in C_c^{\infty}(\Omega)$ with $\psi_{|_{U}}=1$, and $\psi_{|_{\Omega\setminus V}}=0$.
Then the product of $\varphi$ and $\psi$ vanishes,  and we observe
\begin{equation*}
\dual{\AA (\psi y)}{v}
		=
		\dual{\AA (\psi y)}{v}
		-
		\dual{\varphi y}{\psi v}
		=
		\dual{\AA (\psi y)}{v}
		-
		\dual{\AA y}{\psi v}.
\end{equation*}
Similar to the proof of \cref{lem:smooth_multiplier} 
we therefore obtain
	\begin{align*}
		\dual{\AA (\psi y)}{v}
		&=
		\int_\Omega
		y \nabla v^\top A \nabla \psi - v \nabla\psi^\top A \nabla y
		-
		y v \, (b^\top \nabla \psi)
		+
		v y \, (c^\top \nabla \psi)
		\,\d x
		\\&\quad
		\le
		C\norm{y}_{L^6(\Omega)}\norm{\nabla v}_{L^{6/5}(\Omega)}
		+
		C\norm{\nabla y}_{L^2(\Omega)}\norm{v}_{L^{2}(\Omega)}
		+
		C\norm{y}_{L^6(\Omega)}\norm{v}_{L^{6/5}(\Omega)},
		\\&\quad
		\le
		C\norm{y}_{H_0^{1}(\Omega)}\norm{v}_{W^{1,6/5}(\Omega)},
	\end{align*}
	where we have used the embeddings $H_0^1(\Omega)\embeds L^6(\Omega)$ for $y$ as well as $W_0^{1,6/5}(\Omega) \embeds L^2(\Omega)$ in the case $n\le3$ as before.
With \cref{thm:existenceh1} we deduce
\begin{equation*}
\dual{\AA (\psi y)}{v}\le C\norm{f}_{H^{-1}(\Omega)}\norm{v}_{W^{1,6/5}(\Omega)},
\end{equation*}
hence
$\norm{\AA(\psi y)}_{W^{-1,6}(\Omega)}\le C\norm{f}_{H^{-1}(\Omega)}$. Finally, this yields
\begin{equation*}
\norm{\psi y}_{C_0^{0,\alpha}(\Omega)}\le C\norm{f}_{H^{-1}(\Omega)}
\end{equation*}
by \cref{thm:hoelderreg}. 
Note that due to the compact embedding of $C_0^{0,\alpha}(\Omega)\embeds C_0(\Omega)$ the mapping $f\mapsto \psi y$ is compact from $H^{-1}(\Omega)$ into $C_0(\Omega),$
which concludes the proof.
\end{proof}
By a duality argument,
we obtain a regularity result for the adjoint equation.
\begin{lemma}\label{lem:bound_adjoint}
	Let subsets $U,V\subset\Omega$ be given,
	such that $U$ is compact, $V$ is open
	and $U \subset V$.
Let $\mu\in \MM(\Omega)$
with $\supp(\mu) \subset U$
and $\varphi\in C^\infty(\mathbb{R}^n)$ with and $\varphi_{|V}=0$.
Then the very weak solution $p\in W_0^{1,q}(\Omega)$ of 
\[
\AA^\star p = \mu
\] 
fulfills
\begin{equation}\label{eq:bound_adjoint}
	\norm{\varphi p}_{H_0^1(\Omega)}\le C\norm{\mu}_{\MM(\Omega) }.
\end{equation}
Moreover, if $\mu_k\weaklystar\mu$, then $\varphi p_k\to \varphi p$ in $H_0^1(\Omega)$,
where $p_k$ is the very weak solution for the right-hand side $\mu_k$.
\end{lemma}
\begin{proof}
We test the very weak formulation with the solution
$y\in H_0^1(\Omega)\cap C^{0,\alpha}_0(\Omega)$ of $\AA y=\varphi f$ for $f\in L^2(\Omega),$
and obtain
\begin{equation*}
	\int_\Omega p \varphi f \, \d x
	=
	\int_\Omega p \, (\AA y) \, \d x=\int_\Omega y \, \d\mu\le
	\norm{y}_{C(U)}\norm{\mu}_{\MM(\Omega)}
	.
\end{equation*}
Note that the last estimate uses $\supp(\mu) \subset U$.
Applying \cref{lem:hinfboundu} yields
\[
\int_\Omega p \varphi f \, \d x\le C\norm{f}_{H^{-1}(\Omega)}\norm{\mu}_{\MM(\Omega)}.
\]
Since $f \in L^2(\Omega)$ was arbitrary,
this yields $\norm{\varphi p}_{H_0^1(\Omega)}\le C\norm{\mu}_{\MM(\Omega)}$. 

It remains to verify the compactness property.
As in the beginning of the proof,
we have
\begin{equation*}
	\dual{f}{\varphi \, p}_{H_0^1(\Omega)}
	=
	\dual{\mu}{y}_{C(U)}
\end{equation*}
for all $f \in L^2(\Omega)$, $\mu \in \MM(U)$,
where $\AA y = \varphi f$ and $\AA\adjoint p = \mu$ (in the very weak sense).
Using the density of $L^2(\Omega)$
in $H^{-1}(\Omega)$ and \cref{lem:hinfboundu},
this equation extends to all $f \in H^{-1}(\Omega)$.
Therefore, the mapping $\MM(U) \ni \mu \mapsto \varphi \, p \in H_0^1(\Omega)$
is the adjoint of the mapping $H^{-1}(\Omega) \ni f \mapsto y \in C(U)$
from \cref{lem:hinfboundu}.
Now, we can argue as in the proof of \cref{thm:existenceadjoint}.
\end{proof}

\subsection{Solution operator of the obstacle problem}
\label{subsec:solution_operator_obstacle}
In this section, we give some properties of the solution operator
of the variational inequality (VI)
\begin{equation}
	\label{eq:VI:unreg}\tag{$\mathbf{VI}$}
	\text{Find } y\in K \quad \text{such that } \dual{\AA y-u}{v-y} \ge 0\quad\forall v\in K 
\end{equation}
which appears as a constraint in \eqref{eq:P:unreg}.
Here
\begin{equation*}
	K :=\set{v\in H_0^1(\Omega) \given v\ge y_a\quad \text{a.e.~in }\Omega }.
\end{equation*}
We assume the same regularity of $\AA$ and $\Omega$
as in the previous section.
We further suppose that $y_a \le 0$ on $\Gamma$ in the sense of $H^1(\Omega)$
and this implies $K \ne \emptyset$.
First of all,
it is well-known
that
this VI
admits a unique solution $y \in H_0^1(\Omega)$
for each $u \in H^{-1}(\Omega)$,
see \cite[Theorem~II.2.1]{KinderlehrerStampacchia1980}
or \cite[Theorem~4.4]{Troianiello1987}.
The solution operator is denoted by
\[
	S\colon H^{-1}(\Omega)\to H_0^1(\Omega),\quad u\mapsto y
	.
\]
It is known that \eqref{eq:VI:unreg}
is equivalent to the existence of $\xi\in H^{-1}(\Omega)$ such that
\[
	\AA y=u-\xi, \quad
	\xi \in \NN_K(y)
	.
\]
Here, $\NN_K(y)$ is the normal cone of the convex set $K$.

Next, we address Hölder regularity of the solutions.
\begin{theorem}
	\label{thm:solution_VI}
	We assume
	that the obstacle $y_a$
	satisfies $y_a \in H^1(\Omega)$
	and
	$\AA y_a \in L^2(\Omega)$.
	Then, for any $u \in L^2(\Omega)$,
	we have $\AA y, \xi \in L^2(\Omega)$,
	where $y := S(u)$, $\xi := u - \AA y$.
	Moreover,
	$S \colon L^2(\Omega) \to C_0^{0,\alpha}(\Omega)$
	is continuous
	for some $\alpha \in (0,1)$.
\end{theorem}
\begin{proof}
	We can apply \cite[Theorem~4.32]{Troianiello1987}
	and obtain
	the pointwise a.e.\ inequality
	$
	u
	\le
	\AA y
	\le
	\max\set{\AA y_a, u}
	$,
	where $u \in L^2(\Omega)$ is arbitrary
	and $y = S(u)$.
	This implies
	\begin{equation*}
		\norm{\AA y}_{L^2(\Omega)}
		\le
		\norm{u}_{L^2(\Omega)} + \norm{\AA y_a}_{L^2(\Omega)},
	\end{equation*}
	i.e.,
	$u \mapsto \AA y$ is bounded from $L^2(\Omega)$ to $L^2(\Omega)$.

	Now, let a sequence with $u_k \to u$ in $L^2(\Omega)$
	be given.
	We set $y_k := S(u_k)$.
	Since $\AA y_k$ is bounded in $L^2(\Omega)$
	and converges in $H^{-1}(\Omega)$
	to $\AA y$,
	we get
	$\AA y_k \weakly \AA y$
	in $L^2(\Omega)$.

	Since the embedding from $W_0^{1,5/4}(\Omega)$
	into $L^2(\Omega)$
	is compact
	(\cite[Theorem~7.22]{GilbargTrudinger2001}),
	the adjoint embedding
	from $L^2(\Omega)$
	into
	$W^{-1,5}(\Omega) = W_0^{1,5/4}(\Omega)^\star$
	is compact as well (by Schauder's theorem).
	Hence,
	$\AA y_k \to \AA y$
	in $W^{-1,5}(\Omega)$.
	Finally,
	$\AA^{-1}$ is continuous
	from
	$W^{-1,5}(\Omega)$
	to
	$C_0^{0,\alpha}(\Omega)$,
	see \cref{thm:hoelderreg}.
\end{proof}
Using the standard truncation idea due to Stampacchia,
we can also show the Lipschitz continuity of $S$
w.r.t.\ a weaker norm on the forces $u$.
\begin{lemma}
	\label{lem:lipschitz_with_w1p}
	Let $q' > n$ be given.
	Then, there exists $C > 0$
	such that
	\begin{equation*}
		\norm{S(u_1) - S(u_2)}_{L^\infty(\Omega)}
		\le
		C \norm{u_1 - u_2}_{W^{-1,q'}(\Omega)}
		\qquad\forall
		u_1, u_2 \in W^{-1,q'}(\Omega).
	\end{equation*}
\end{lemma}
\begin{proof}
	Let $u_1, u_2 \in W^{-1,q'}(\Omega)$ be given
	and set $y_j := S(u_j)$ for $j = 1,2$.
	For $k > 0$, we define $\hat y_k := (y_1 - y_2 - k)^+$.
	Due to
	$
		y_1 - \hat y_k
		=
		\min(y_2 + k , y_1 )
		\ge
		y_a
	$,
	we can test the VIs for $y_1$ and $y_2$ with $y_1 - \hat y_k$ and $y_2 + \hat y_k$,
	respectively.
	Adding the resulting inequalities leads to
	\begin{equation*}
		\dual{\AA (y_1 - y_2)}{\hat y_k}
		\le
		\dual{u_1 - u_2}{\hat y_k}
		\qquad\forall k > 0.
	\end{equation*}
	Now, we can use the arguments of \cite[Lemma~2.8]{Troianiello1987}
	(see also the remark following this lemma)
	to obtain
	\begin{equation*}
		\esssup (y_1 - y_2)
		\le
		C \norm{u_1 - u_2}_{W^{-1,p'}(\Omega)}
		.
	\end{equation*}
	Note that we can avoid the $L^2(\Omega)$-norm
	of $y_1 - y_2$
	on the right-hand side of this estimate
	due to the Lipschitz-continuity
	of $S$ from $H^{-1}(\Omega)$ to $H_0^1(\Omega)$.
	By interchanging the roles of $y_1$ and $y_2$,
	we arrive at the claimed estimate.
\end{proof}

An important property is the monotonicity
\begin{equation}
	\label{eq:monotonicity}
	u_1 \le u_2
	\quad\Rightarrow\quad
	S(u_1) \le S(u_2),
\end{equation}
see, e.g.,
\cite[Corollary, p.~242]{Troianiello1987}.

From the seminal work of Mignot,
we get the directional differentiability of
the mapping
$S \colon H^{-1}(\Omega) \to H_0^1(\Omega)$,
see \cite[Théorème~3.3]{Mignot1976}.
\begin{theorem}
	\label{thm:direc_diff}
	The solution operator $S \colon H^{-1}(\Omega) \to H_0^1(\Omega)$
	is directionally differentiable at all points $\bar u \in H^{-1}(\Omega)$.
	The directional derivative $z := S'(\bar u; h) \in H_0^1(\Omega)$
	in direction $h \in H^{-1}(\Omega)$
	is given by the unique solution of the VI
	\begin{equation}
		\label{eq:direc_diff}
		z \in \KK(\bar u),
		\qquad
		\dual{\AA z - h}{v - z} \ge 0
		\quad\forall v \in \KK(\bar u).
	\end{equation}
	Here,
	\begin{equation*}
		\KK(\bar u) := \TT_K(\bar y) \cap \bar\xi\anni,
	\end{equation*}
	where $\bar y = S(\bar u)$ and $\bar\xi := \bar u - \AA \bar y$
	are the associated state and multiplier,
	respectively.
	For the critical cone $\KK(\bar u)$,
	we have the representation
	\begin{equation}
		\label{eq:crit_cone}
		\KK(\bar u)
		=
		\set[\big]{
			v \in H_0^1(\Omega)
			\given
			v \ge 0 \text{ q.e.\ on } \set{\bar y = y_a}
			\;\text{and}\;
			v = 0 \text{ q.e.\ on } \qsupp(\bar \xi)
		}.
	\end{equation}
\end{theorem}
The formula for the critical cone
involving the quasi-support of $\bar\xi$
can be found in
\cite[Lemma~3.1]{Wachsmuth2013:2}.

Since $S \colon H^{-1}(\Omega) \to H_0^1(\Omega)$ is Lipschitz continuous,
we obtain that $S$ is even Hadamard differentiable,
i.e.,
$(S(\bar u + t_k h_k) - S(\bar u)) / t_k \to S'(\bar u; h)$
if $h_k \to h$ in $H^{-1}(\Omega)$
and $t_k \searrow 0$.

We note that the monotonicity \eqref{eq:monotonicity}
implies
\begin{equation}
	\label{eq:mon_deriv}
	h_1 \le h_2
	\quad\Rightarrow\quad
	S'(u; h_1 ) \le S'(u; h_2)
\end{equation}
for the directional derivative.

The next lemma shows that the difference quotients
converge uniformly on the set where $\bar y$ has a positive
distance from the lower bound $y_a$.
\begin{lemma}
	\label{lem:better_differentiability}
	Let the assumptions of \cref{thm:solution_VI}
	and, additionally,
	$y_a \in C(\bar\Omega)$
	be satisfied.
	For $\bar u \in L^2(\Omega)$,
	we define the state $\bar y := S(\bar u)$
	and the set
	\begin{equation*}
		\hat\Omega := \set{ \bar y \ge y_a + \sigma}
		,
	\end{equation*}
	where $\sigma > 0$ is arbitrary
	and we use the continuous representatives of $\bar y$ and $y_a$.
	\begin{enumerate}[label=(\alph*)]
		\item\label{item:better_differentiability_1}
			For an arbitrary $q' > n$, there exists a constant $C > 0$, such that
			\begin{align*}
				\norm{S'(\bar u; h_1) - S'(\bar u; h_2)}_{L^\infty(\hat\Omega)}
				&\le
				C \norm{h_1 - h_2}_{W^{-1,q'}(\Omega)}
				\qquad\forall h_1, h_2 \in W^{-1,q'}(\Omega).
			\end{align*}
		\item\label{item:better_differentiability_3}
			Let $\tilde\varphi \in C^\infty(\R^n)$ be given such that $\tilde\varphi$ vanishes on
			a neighborhood of $\hat\Omega$. Then,
			\begin{equation*}
				\norm{S'(\bar u; \tilde\varphi h_1) - S'(\bar u; \tilde\varphi h_2)}_{L^\infty(\hat\Omega)}
				\le
				C \norm{h_1 - h_2}_{H^{-1}(\Omega)}
				\qquad\forall h_1, h_2 \in H^{-1}(\Omega).
			\end{equation*}
		\item\label{item:better_differentiability_2}
			Let sequences $\seq{h_k}_{k \in \N} \subset L^2(\Omega)$
			and $\seq{t_k}_{k \in \N} \subset \R^+$
			be given such that
			$h_k \to h$ in $W^{-1,q'}(\Omega)$
			and
			$t_k \searrow 0$.
			We define the perturbed states $y_k := S(\bar u + t_k h_k)$.
			Then,
			the difference quotients $(y_k - \bar y)/t_k$
			converge
			towards $S'(\bar u; h)$
			uniformly on the set $\hat \Omega$
			as $k \to \infty$.
			In particular,
			$S'(\bar u; h)$ is continuous on $\hat \Omega$.
	\end{enumerate}
\end{lemma}
\begin{proof}
	Note that we get continuity of $\bar y$
	from \cref{thm:solution_VI}
	and
	$y_a$ is continuous by assumption.
	Thus, the sets $\hat\Omega$ and
	\begin{equation*}
		\hat\Omega_2 := \set{ \bar y \le y_a + \sigma/2}
	\end{equation*}
	are closed.
	In the sequel, we are going to apply the regularity result
	\cref{lem:hinfboundu} with $U = \hat\Omega$.
	Therefore, we fix a function $\varphi \in C^\infty(\R^n)$,
	such that $0 \le \varphi \le 1$ on $\R^n$,
	$\varphi = 0$ on a neighborhood $V$ of $U = \hat\Omega$
	and $\varphi = 1$ on $\hat\Omega_2$.
	Note that this is possible,
	since
	the sets
	$\hat\Omega, \hat\Omega_2$ have a positive distance.

	We start with \ref{item:better_differentiability_1}.
	Let $h_1, h_2 \in W^{-1,q'}(\Omega)$ be given.
	We define the functional
	\begin{equation*}
		\hat\xi
		:=
		\parens[\big]{ h_1 - \AA S'(\bar u; h_1) }
		-
		\parens[\big]{ h_2 - \AA S'(\bar u; h_2) }
		\in \KK(\bar u)\polar 
		,
	\end{equation*}
	see \eqref{eq:direc_diff}.
	Note that
	\begin{equation*}
		\norm{\hat\xi}_{H^{-1}(\Omega)}
		\le
		\norm{h_1 - h_2}_{H^{-1}(\Omega)}
		.
	\end{equation*}
	For arbitrary $v \in H_0^1(\Omega)$,
	we have $(1-\varphi) v = 0$ q.e.\ on $\hat\Omega_2 \supset \set{\bar y = y_a}$.
	Thus, $\pm (1-\varphi) v \in \KK(\bar u)$.
	Thus,
	$\dual{(1-\varphi) \hat\xi}{v} = 0$,
	i.e.,
	$\hat\xi = \varphi \hat\xi$.
	By \cref{thm:hoelderreg,lem:hinfboundu}, we get
	\begin{align*}
		\norm{S'(\bar u; h_1) - S'(\bar u; h_2)}_{L^\infty(\hat\Omega)}
		&=
		\norm{\AA^{-1}(\varphi \hat\xi - h_1 + h_2)}_{L^\infty(\hat\Omega)}
		\\&
		\le
		C \parens*{ \norm{\hat\xi}_{H^{-1}(\Omega)} + \norm{h_1 - h_2}_{W^{-1,q'}(\Omega)} }
		\\&
		\le
		C \parens*{ \norm{h_1 - h_2}_{H^{-1}(\Omega)} + \norm{h_1 - h_2}_{W^{-1,q'}(\Omega)} }
	\end{align*}
	and the assertion follows.

	The proof of \ref{item:better_differentiability_3} is very similar.
	One just has to replace \cref{thm:hoelderreg}
	by another application of \cref{lem:hinfboundu}.

	Next, we show \ref{item:better_differentiability_2}.
	Since $S \colon W^{-1,q'}(\Omega) \to L^\infty(\Omega)$
	is Lipschitz by \cref{lem:lipschitz_with_w1p},
	there exists $N \in \N$ such that
	\begin{equation*}
		y_k(x) > y_a(x)
		\qquad\forall x \in \Omega \setminus \hat \Omega_2, k \ge N.
	\end{equation*}
	Hence, the associated multiplier $\xi_k := (\bar u + t_k h_k) - \AA y_k \in L^2(\Omega)$
	is supported on the set $\hat \Omega_2$.
	Now, we consider the difference quotient of the multipliers
	$\hat\xi_k := (\xi_k - \bar\xi)/t_k$,
	where $\bar\xi := \bar u - \AA \bar y$.
	This difference quotient is supported on $\hat\Omega_2$
	and converges in $H^{-1}(\Omega)$
	towards $\hat\xi := h - \AA S'(\bar u; h)$,
	see \cref{thm:direc_diff}.
	This implies
	$\hat\xi_k = \varphi \hat\xi_k \to \varphi \hat\xi = \hat\xi$.

	Thus, we can apply
	\cref{lem:hinfboundu}
	and obtain
	\begin{equation*}
		\norm*{
			\AA^{-1} \hat\xi_k
			-
			\AA^{-1} \hat\xi
		}_{L^\infty(\hat\Omega)}
		=
		\norm{
			\AA^{-1} \parens*{\varphi (\hat\xi_k - \hat\xi)}
		}_{L^\infty(\hat\Omega)}
		\le
		\norm{\hat\xi_k - \hat\xi}_{H^{-1}(\Omega)}
		\to
		0
		.
	\end{equation*}
	Together with
	\begin{equation*}
		\frac{y_k - \bar y}{t_k}
		=
		\AA^{-1} h + \AA^{-1} \hat\xi_k
		\qquad\text{and}\qquad
		S'(\bar u; h)
		=
		\AA^{-1} h + \AA^{-1} \hat\xi
		,
	\end{equation*}
	this shows the claim.
\end{proof}

A well-known property
of the obstacle problem with a lower bound
is the pointwise convexity of the solution operator.
This renders the state constraint convex w.r.t.\ the control $u$
and will become important for our analysis
of the optimal control problem.
\begin{lemma}
	\label{lem:pw_convex}
	Let $u_1, u_2 \in H^{-1}(\Omega)$ and $\alpha \in (0,1)$ be given.
	Then,
	\begin{equation*}
		S(\alpha u_1 + (1-\alpha) u_2)
		\le
		\alpha S(u_1)
		+
		(1-\alpha) S(u_2)
		\qquad\text{a.e.\ in $\Omega$}.
	\end{equation*}
\end{lemma}
\begin{proof}
	For convenience, we reproduce the short proof from
	\cite[Lemme~4.1]{Mignot1976}.

	We set
	$u_3 := \alpha u_1 + (1-\alpha) u_2$
	and
	$y_i := S(u_i)$ for $i = 1,\ldots, 3$.
	We have to show that
	$w := \parens{y_3 - \alpha y_1 - (1-\alpha) y_2 }^+$
	is zero.
	This definition directly implies
	$w \ge 0$
	and
	$y_3 - w = \min\braces{y_3, \alpha y_1 + (1-\alpha) y_2} \ge y_a $.
	Thus,
	we obtain
	\begin{align*}
		\dual{\AA y_1 - u_1}{w}
		&=
		\dual{\AA y_1 - u_1}{y_1 + w - y_1}
		\ge
		0,
		\\
		\dual{\AA y_2 - u_2}{w}
		&=
		\dual{\AA y_2 - u_2}{y_2 + w - y_2}
		\ge
		0,
		\\
		\dual{\AA y_3 - u_3}{-w}
		&=
		\dual{\AA y_3 - u_3}{y_3-w - y_3}
		\ge
		0
		.
	\end{align*}
	We multiply the first two inequalities by $\alpha$ and $(1-\alpha),$ respectively.
	Adding the three resulting inequalities yields
	\begin{equation*}
		\dual{\AA (y_3 - \alpha y_1 - (1-\alpha) y_2)}{w}
		\le
		0.
	\end{equation*}
	The structure of the differential operator gives
	\begin{equation*}
		\dual{\AA v^+}{v^+}
		=
		\dual{\AA v}{v^+}
		\qquad\forall v \in H_0^1(\Omega)
	\end{equation*}
	and together with the coercivity of $\AA$
	we obtain $w = 0$.
\end{proof}
From this pointwise convexity,
we obtain two inequalities for the directional derivative
by the usual arguments.
\begin{corollary}
	\label{cor:conv_direc_deriv}
	Let $u, h, h_2 \in H^{-1}(\Omega)$ be given.
	Then,
	\begin{align*}
		S(u) + S'(u; h) &\le S(u + h)
		\qquad\text{a.e.\ in $\Omega$},
		\\
		S'(u; h + h_2) &\le S'(u; h) + S'(u; h_2)
		\qquad\text{a.e.\ in $\Omega$}.
	\end{align*}
\end{corollary}
\begin{proof}
	For any $t \in (0,1)$ we have
	\begin{equation*}
		S(u + t h) = S( (1-t) u + t (u + h) )
		\le
		(1-t) S(u) + t S(u + h).
	\end{equation*}
	Now, we subtract $S(u),$ divide by $t > 0$
	and pass to the limit $t \searrow 0$
	to arrive at the first assertion.
	The second assertion follows similarly by considering
	\begin{equation*}
		2 \, \parens[\big]{S(u + t \, (h + h_2)) - S(u)}
		\le
		S(u + t \, (2 h)) - S(u)
		+
		S(u + t \, (2 h_2)) - S(u)
		,
	\end{equation*}
	dividing by $t > 0,$ passing to the limit $t \searrow 0$
	and using the positive homogeneity of $S'(u; \cdot)$.
\end{proof}

\subsection{Optimal control problem}
\label{subsec:oc}
We will now discuss the optimal control problem \eqref{eq:P:unreg}. To this end, let us collect
all the assumptions which have been made in the previous preliminary results.
Additionally,
we make further assumptions
concerning the optimal control problem,
in particular we will assume the existence of a Slater point, from which we will eventually deduce existence of a Lagrange multiplier associated with the state constraints. 

\begin{assumption}\leavevmode
	\label{asm:big_list}
	\begin{enumerate}[label=(\roman*)]
		\item
			\label{item:asm_domain}
			The domain $\Omega \subset \R^n,$ $n \in \set{2,3}$,
			is bounded and
			satisfies the uniform exterior cone condition,
			see \cite[p.~205]{GilbargTrudinger2001}.
		\item
			\label{item:asm_diff_op}
			The differential operator $\AA$ is given as in \eqref{eq:diff_op},
			such that
			\eqref{eq:regularity_coefficients}--\eqref{eq:coercivity}
			are satisfied.
		\item
			\label{item:asm_obstacle}
			The obstacle $y_a \in H^1(\Omega) \cap C(\bar\Omega)$ satisfies
			$y_a \le 0$ on $\Gamma$ in the sense
			$\max\set{y_a, 0} \in H_0^1(\Omega)$
			and
			$\AA y_a \in L^2(\Omega)$.
		\item
			\label{item:asm_state_const}
			The state constraint has the regularity $y_b \in  C(\bar\Omega)$ and satisfies
			$y_b>0$ on $\Gamma$.
		\item
			\label{item:asm_Uad}
			The control set $\Uad \subset L^2(\Omega)$
			is convex, closed, and non-empty.
		\item
			\label{item:asm_objective}
			The objective $J \colon H_0^1(\Omega) \times L^2(\Omega) \to \R$
			is assumed to be
			continuously Fréchet-differentiable and bounded from below.
			We require that $J$ is
			sequentially lower semi-continuous
			w.r.t.\ to the strong topology in $H_0^1(\Omega)$ and the weak topology in $L^2(\Omega)$, that is
			$J(y,u) \le \liminf_{k \to \infty} J(y_k, u_k)$
			for all sequences $\seq[\big]{(y_k, u_k)}_{k \in \N} \subset H_0^1(\Omega) \times L^2(\Omega)$
			satisfying $y_k \to y$ in $H_0^1(\Omega)$ and $u_k \weakly u$ in $L^2(\Omega)$.
			Finally, we assume that $J$ is coercive w.r.t.\ the second variable on the feasible set $\Uad$,
			that is
			the boundedness of $\seq{u_k}_{k \in \N}$ in $L^2(\Omega)$ follows from the boundedness of $\seq{J(y_k, u_k)}_{k \in \N}$
			for all sequences $\seq[\big]{(y_k, u_k)}_{k \in \N} \subset H_0^1(\Omega) \times \Uad$.
		\item
			\label{item:asm_slater}
			There exists a Slater point $\hat u \in \Uad$
			with $S(\hat u) \le y_b - \tau$
			on $\Omega$
			for some $\tau > 0$.
	\end{enumerate}
\end{assumption}

Due to the Slater condition,
we have
$y_a \le y_b - \tau$
on $\Omega$.
In the case without control constraints,
this apparently weaker condition
already implies the Slater condition.
\begin{lemma}
	\label{lem:slater_from_distance}
	We assume \cref{asm:big_list}~\ref{item:asm_domain}--\ref{item:asm_state_const}.
	If, additionally, $y_a \le y_b - \tau$ on $\Omega$ for some $\tau > 0$, $y_a<0$ on $\Gamma$,
	and $\Uad = L^2(\Omega)$,
	then
	\cref{asm:big_list}~\ref{item:asm_slater}
	follows.
\end{lemma}
\begin{proof}
	Since $y_a, y_b$ are continuous, $y_b>0>y_a$ on $\Gamma,$ and $y_b-y_a\ge \tau>0$, it is possible to construct an arbitrarily smooth function $\tilde y\ge y_a$ with $\tilde y=0$ on $\Gamma$ and positive distance to $y_b$.
	Now we define $\tilde u:=\AA\tilde y\in W^{-1,q'}(\Omega)$, $q'>n$.
	By a density argument, we can smooth $\tilde u$ and construct $\hat u\in L^2(\Omega)$ such that,
	with $\hat y=S(\hat u)$, \cref{lem:lipschitz_with_w1p}
	guarantees
	$\norm{\tilde y-\hat y}_{L^\infty(\Omega)}\le \varepsilon$ for any fixed, positive $\varepsilon$.
	Therefore, choosing $\hat u$ corresponding to $\varepsilon>0$ small enough,
	$\hat y$ has positive distance to $y_b$, meaning that $\hat u$ fulfills the Slater point property.
\end{proof}
Note that one cannot expect $\hat u\in \Uad$ if $\Uad\neq L^2(\Omega)$.

If the admissible set $\Uad$ has a minimal point,
i.e., $u_b \in \Uad$ with $u \ge u_b$ for all $u \in \Uad$,
then there exists a Slater point if and only if
$u_b$ is a Slater point.
This follows easily from the monotonicity of $S$,
see \eqref{eq:monotonicity}.

From now on,
we will always assume that \cref{asm:big_list}
is satisfied.

The existence of the Slater point $\hat u$ will not only be used to show optimality conditions. As a side effect, it also
guarantees the existence of solutions.
\begin{theorem}
	\label{thm:existence}
	There exists at least one globally optimal control $\bar u\in \Uad$ to \eqref{eq:P:unreg} with associated optimal state $\bar y\in H_0^1(\Omega)$.
\end{theorem}
\begin{proof}
	Due to \cref{asm:big_list}~\ref{item:asm_slater}, there exists a
	feasible pair $(\hat y, \hat u)$.
	Then, we infer the existence of an optimal solution $(\bar y,\bar u)$ to Problem \eqref{eq:P:unreg} by standard arguments.
\end{proof}

Note that due to the nonlinearity of the solution operator $S$,
one cannot show uniqueness of the solution.

\section{Primal optimality conditions}
\label{sec:primal_systems}
In this section,
we address necessary optimality conditions for \eqref{eq:P:unreg}
which do not involve dual quantities.
We start by masking the state constraint as a convex control constraint
via \cref{lem:pw_convex}.
\begin{lemma}
	\label{lem:mask_state_constraint}
	We define
	\begin{align*}
		\Ustate
		&:=
		\set{
			u \in L^2(\Omega)
			\given
			S(u) \le y_b
			\;\text{in }\Omega
		}
		,
		\\
		\Ueff
		&:=
		\Uad \cap \Ustate
		=
		\set{
			u \in \Uad
			\given
			S(u) \le y_b
			\;\text{in }\Omega
		}.
	\end{align*}
	The sets $\Ustate, \Ueff \subset L^2(\Omega)$
	are closed and convex.
\end{lemma}
\begin{proof}
	The convexity follows from \cref{lem:pw_convex}
	and the closedness from the continuity of $S\colon L^2(\Omega) \to C_0^{0,\alpha}(\Omega)$,
	see \cref{thm:solution_VI}.
\end{proof}
Using this result,
we can reformulate \eqref{eq:P:unreg}
and obtain the equivalent problem
\begin{equation}
	\label{eq:P:modified}
	\begin{aligned}
		\text{Minimize}&\quad J(S(u),u) \\
		\text{with respect to}&\quad (y,u)\in H_0^1(\Omega)\times L^2(\Omega)\\
		\text{such that}& \quad u \in \Ueff.
	\end{aligned}
\end{equation}
Since the set $\Ueff$ is closed and convex,
we could apply \cite[Theorem~1.1]{Wachsmuth2014:2}
to obtain a system of C-stationarity.
This system contains the normal cone to $\Ueff$
in $L^2(\Omega)$
and it is not immediately clear
how to evaluate this normal cone.
After we have characterized the normal cone
in \cref{lem:normal_cone_via_strong},
we comment on this approach at the end of \cref{sec:strong_stat}.

Another possibility is to use the directional differentiability of $S$
to arrive at a primal optimality system.
\begin{lemma}
	\label{lem:B_stat_first}
	Let $\bar u$ be locally optimal for \eqref{eq:P:unreg}
	with associated state $\bar y=S(\bar u)$.
	Then,
	\begin{equation*}
		\dual{J_y(\bar y, \bar u)}{S'(\bar u; h)} + \innerprod{J_u(\bar y,\bar u)}{h} \ge 0
		\qquad\forall
		h \in \TT_{\Ueff}(\bar u)
	\end{equation*}
	is necessary for the optimality of $\bar u$.
\end{lemma}
Here,
$J_y(\bar y,\bar u) \in H^{-1}(\Omega)$
and
$J_u(\bar y,\bar u) \in L^2(\Omega)$
are the partial derivatives of $J$
w.r.t.\ $y$ and $u$
evaluated at $(\bar y, \bar u)$.
\begin{proof}
	For any $h \in \RR_{\Ueff}(\bar u)$,
	we have $u + t h \in \Ueff$ for $t > 0$ small enough.
	Thus,
	\begin{equation*}
		J(S(u+t h),u+t h) - J(S(u),u) \ge 0.
	\end{equation*}
	Dividing by $t > 0$ and passing to the limit $t \searrow 0$,
	we obtain
	\begin{equation*}
		\dual{J_y(\bar y,\bar u)}{S'(\bar u; h)} + \innerprod{J_u(\bar y,\bar u)}{h} \ge 0
		\qquad\forall
		h \in \RR_{\Ueff}(\bar u)
		.
	\end{equation*}
	Since $\RR_{\Ueff}(\bar u)$ is dense in $\TT_{\Ueff}(\bar u)$
	and since the left-hand side of the inequality is continuous
	w.r.t.\ $u \in L^2(\Omega)$,
	the claim follows.
\end{proof}

Our next goal is the characterization
of the tangent cone of $\Ueff$.
We start by the investigation of the tangent cone of $\Ustate$.

\begin{theorem}
	\label{thm:tangent_Ustate}
	Let $\bar u \in \Ustate$ be given.
	Then,
	\begin{equation}
		\label{eq:tangent_Ustate}
		\TT_{\Ustate}(\bar u)
		=
		\set{
			h \in L^2(\Omega)
			\given
			S'(\bar u; h) \le 0 \text{ everywhere on } \Omega_b
		}
		,
	\end{equation}
	where $\Omega_b := \set{\bar y = y_b}$ with $\bar y = S(\bar u).$
\end{theorem}
Note that $S'(\bar u; h)$ is continuous in the neighborhood
$\set{\bar y \ge y_a + \zeta/2}$
of
$\Omega_b$
via \cref{lem:better_differentiability}.
Thus, the inequality can be understood in an
``everywhere''-sense.
\begin{proof}
	``$\subset$'':
	Let $h \in \TT_{\Ustate}(\bar u)$ be given.
	Then, there are sequences
	$\seq{u_k}_{k \in \N} \subset \Ustate$, $\seq{t_k}_{k \in \N} \subset \R^+$
	such that
	$u_k \to \bar u$, $t_k \searrow 0$
	and $h_k := (u_k - \bar u) / t_k \to h$ in $L^2(\Omega)$.
	We define $y_k := S(u_k) = S(\bar u + t_k h_k)$.
	Then,
	$0 \ge \frac{y_k - \bar y}{t_k}$
	holds everywhere on $\Omega_b$.
	Due to \cref{lem:better_differentiability}~\ref{item:better_differentiability_2},
	the right-hand side
	converges uniformly on $\Omega_b$ towards $S'(\bar u; h)$.
	Hence, $S'(\bar u; h) \le 0$ on $\Omega_b$.

	``$\supset$'':
	Let $h \in L^2(\Omega)$ with $S'(\bar u; h) \le 0$ on $\Omega_b$
	be given.
	In case $\Omega_b = \emptyset$,
	the continuous functions $\bar y$ and $y_b$ have a positive distance.
	Therefore,
	the claim follows from the continuity of $S \colon L^2(\Omega) \to C_0^{0,\alpha}(\Omega)$,
	see \cref{thm:solution_VI}.

	Otherwise, let $\seq{t_k}_{k \in \N} \subset \R^+$ with $t_k \searrow 0$ be given.
	Due to the continuity of $S$ from $L^2(\Omega)$
	to $C_0^{0,\alpha}(\Omega)$,
	there exists a sequence $\seq{s_k}_{k \in \N} \subset \R^+$, $s_k \searrow 0$,
	such that
	\begin{equation*}
		S(\bar u + t_k h) \le S(\bar u) + s_k
		\qquad
		\text{on } \Omega.
	\end{equation*}
	W.l.o.g.\ we assume $s_k \le \zeta/3$.
	Therefore,
	the sets
	\begin{equation*}
		\Omega_k :=
		\set{\bar y \ge y_b - s_k}
		,
		\quad
		\hat\Omega :=
		\set{\bar y \ge y_a + \zeta / 3}
	\end{equation*}
	satisfy
	$\Omega_b \subset \Omega_k \subset \hat\Omega$
	for all $k \in \N$.
	We define the scalar sequence
	\begin{equation*}
		d_k
		:=
		\sup\set{
			S'(\bar u; h)(x)
			\given
			x \in \Omega_k
		}
		\ge
		0.
	\end{equation*}
	We claim that $d_k \searrow 0$.
	Indeed, otherwise we would get a sequence $\seq{x_k}_{k \in \N}$
	with $x_k \in \Omega_k$ and $S'(\bar u; h)(x_k) \ge \varepsilon > 0$.
	This sequence has accumulation points
	and due to continuity, all accumulation points $\bar x$
	satisfy $S'(\bar u; h)(\bar x) \ge \varepsilon > 0$ and $\bar y(\bar x)  \ge y_b(\bar x)$, i.e., $\bar x \in \Omega_b$.
	This is a contradiction to $S'(\bar u; h) \le 0$ on $\Omega_b$.

	Due to \cref{lem:better_differentiability}~\ref{item:better_differentiability_2},
	\begin{equation*}
		r_k
		:=
		\norm*{
			\frac{ S(\bar u + t_k h) - \bar y }{t_k}
			-
			S'(\bar u; h)
		}_{C(\hat \Omega)}
		\searrow
		0
		.
	\end{equation*}
	Now we have
	\begin{align*}
		S(\bar u + t_k h)
		&
		\le
		S(\bar u) + t_k S'(\bar u; h) + t_k r_k
		\le
		y_b + t_k \, \parens{ d_k + r_k }
		&& \text{on }\Omega_k \subset \hat \Omega,
		\\
		S(\bar u + t_k h)
		&
		\le
		S(\bar u) + s_k
		\le
		y_b
		&& \text{on }\Omega \setminus \Omega_k.
	\end{align*}
	Next, we use the Slater point $\hat u \in L^2(\Omega)$,
	i.e., $S(\hat u) \le y_b - \tau$ for some $\tau > 0$.
	We set
	\begin{equation*}
		h_k
		:=
		(1-\alpha_k) h
		+
		\frac{\alpha_k}{t_k} \, (\hat u - \bar u)
		,
		\qquad
		\alpha_k := \frac{d_k + r_k}{\tau} t_k
		.
	\end{equation*}
	From $\alpha_k/t_k \to 0$ we get $h_k \to h$ in $L^2(\Omega)$.
	Moreover,
	for $k$ large enough we
	have $\alpha_k \in (0,1)$
	and
	via \cref{lem:pw_convex}
	we obtain
	\begin{align*}
		S(\bar u + t_k h_k)
		&=
		S\parens[\big]{
			(1-\alpha_k) \, (\bar u + t_k h)
			+
			\alpha_k \hat u
		}
		\\
		&\le
		(1-\alpha_k) S(\bar u + t_k h)
		+
		\alpha_k S(\hat u)
		\\
		&\le
		(1-\alpha_k) \, ( y_b + t_k \, (d_k + r_k))
		+
		\alpha_k \, (y_b - \tau)
		\\
		&=
		y_b
		+
		(1-\alpha_k) t_k \, (d_k + r_k)
		-\alpha_k \tau
		\\
		&\le
		y_b
		+
		t_k \, (d_k + r_k)
		-\alpha_k \tau
		=
		y_b
		\qquad\text{on } \Omega.
	\end{align*}
	This shows $\bar u + t_k h_k \in \Ustate$.
	Together with
	$h_k \to h$ in $L^2(\Omega)$
	we get $h \in \TT_{\Ustate}(\bar u)$.
\end{proof}

Using the Slater point again,
we can characterize
the tangent cone and normal cone
to
$\Ueff = \Uad \cap \Ustate$.

\begin{theorem}
	\label{thm:cones_Ueff}
	Let $\bar u \in \Ueff$ be given.
	Then,
	\begin{align*}
		\TT_{\Ueff}(\bar u)
		&=
		\TT_{\Uad}(\bar u)
		\cap
		\TT_{\Ustate}(\bar u)
		,
		&
		\NN_{\Ueff}(\bar u)
		&=
		\NN_{\Uad}(\bar u)
		+
		\NN_{\Ustate}(\bar u)
		.
	\end{align*}
\end{theorem}
\begin{proof}
	Due to the continuity of $S \colon L^2(\Omega) \to C_0^{0,\alpha}(\Omega)$,
	see \cref{thm:solution_VI},
	the Slater point $\hat u$ is an interior point of $\Ustate$
	and belongs to $\Uad$.
	Hence, we can apply
	the sum rule of convex analysis \cite[Corollary~16.38]{BauschkeCombettes2011}
	to the indicator function $\delta_{\Ueff} = \delta_{\Uad} \cap \delta_{\Ustate}$
	and obtain
	\begin{equation*}
		\NN_{\Ueff}(\bar u)
		=
		\partial \delta_{\Ueff}(\bar u)
		=
		\partial \delta_{\Uad}(\bar u)
		+
		\partial \delta_{\Ustate}(\bar u)
		=
		\NN_{\Uad}(\bar u)
		+
		\NN_{\Ustate}(\bar u).
	\end{equation*}
	The tangent cone can be obtained by polarization via the bipolar theorem.
\end{proof}
Together with \cref{lem:B_stat_first},
we obtain the following optimality condition.

\begin{theorem}
	\label{thm:B_stat_second}
	Every locally optimal solution $\bar u$ of \eqref{eq:P:unreg}
	satisfies
	\begin{equation}
		\label{eq:B_stat_second}
		\dual{J_y(\bar y,\bar u)}{S'(\bar u; h)} + \innerprod{J_u(\bar y,\bar u)}{h} \ge 0
		\qquad\forall
		h \in \TT_{\Uad}(\bar u),
		S'(\bar u; h) \le 0 \text{ on } \Omega_b,
	\end{equation}
	where
	$\Omega_b := \set{\bar y = y_b}$
	and
	$\bar y := S(\bar u)$.
\end{theorem}
Although we have derived a characterization of the tangent cone
of $\Ustate$, see \cref{thm:tangent_Ustate},
this cannot be employed
to obtain an expression for the normal cone,
due to the nonlinearity of $S'(\bar u; \cdot)$.
Even if an explicit formula for this normal cone would be available,
the primal optimality condition \eqref{eq:B_stat_second}
cannot be turned directly into a dual optimality condition,
since the left-hand side in \eqref{eq:B_stat_second}
depends nonlinearly on $h$.
We mention that a formula for $\NN_{\Ustate}(\bar u)$
will be given in
\cref{lem:normal_cone_via_strong} below.

\section{Dual optimality conditions via regularization}
\label{sec:regularization}
In this section,
we are going to derive
optimality conditions
which include multipliers
via a regularization procedure.
We will prove the following theorem.
\begin{theorem}
	\label{thm:c_stationarity}
	Every local solution $(\bar y, \bar u)$
	of \eqref{eq:P:unreg}
	is C-stationary,
	i.e., there exist
	multipliers
	$p \in W_0^{1,q}(\Omega)$,
	$\mu \in H^{-1}(\Omega)$,
	$\nu \in \MM(\Omega)^+$,
	$\lambda \in L^2(\Omega)$
	such that
	$p \in H^1(\hat\Omega_a)$
	for some open $\hat\Omega_a \supset \set{\bar y = y_a}$
	and
	such that the system
	\begin{subequations}
		\label{eq:c_stat}
		\begin{align}
			\label{eq:c_stat_1}
			\AA\adjoint p + J_y(\bar y,\bar u) + \nu + \mu &= 0,
			\\
			\label{eq:c_stat_2}
			J_u(\bar y,\bar u) + \lambda - p &= 0,
			\\
			\label{eq:c_stat_3}
			p &= 0 \text{ q.e.\ on } \qsupp(\bar \xi),
			\\
			\label{eq:c_stat_4}
			\dual{\mu}{v}_{H_0^1(\Omega)}
			&= 0 \quad\forall v \in H_0^1(\Omega), v = 0 \text{ q.e.\ on } \set{\bar y = y_a},
			\\
			\label{eq:c_stat_45}
			\dual{\mu}{\Phi p}_{H_0^1(\Omega)} &\ge 0
			\quad
			\forall \Phi \in W^{1,\infty}(\Omega)^+,\, \Phi_{|_{\Omega \setminus \hat\Omega_a}}=0,
			\\
			\label{eq:c_stat_5}
			\supp(\nu) &\subset \Omega_b,
			\\
			\label{eq:c_stat_6}
			\lambda &\in \NN_{\Uad}(\bar u)
		\end{align}
	\end{subequations}
	is satisfied.
	Here,
	the adjoint equation is to be understood in the very weak sense,
	see \eqref{eq:adjveryweak2},
	and $q \in (1, n / (n - 1))$ can be chosen arbitrarily.
	Note that $\Phi p\in H^1_0(\Omega)$ even though $p$ itself is only in $W_0^{1,p}(\Omega)$.
\end{theorem}
The proof of this theorem
is divided into several steps,
which will be addressed in the remaining part
of this section:
\begin{itemize}
	\item
		\cref{subsec:reg_aux}:
		Existence of solutions and optimality condition
		for regularized problems.
	\item
		\cref{subsec:boundedness_mult}:
		Boundedness of the multipliers of the regularized optimality system.
	\item
		\cref{subsec:passage_limit}:
		Passage to the limit in the optimality system.
\end{itemize}

Throughout the remaining part of this section,
we fix a local solution
$(\bar y, \bar u)$ of \eqref{eq:P:unreg}.

\subsection{Regularized problems}
\label{subsec:reg_aux}
In order to derive an optimality condition for problem \eqref{eq:P:unreg},
we consider a regularization of the state constraint by penalization of any violation of the constraints, see \cite{itokun03}.
Clearly,
other regularization approaches would be viable as well,
i.e., a regularization of the obstacle problem.
For a regularization parameter $\gamma > 0$,
define the regularized problem
\begin{equation}
	\label{eq:P:reg}
	\tag{$\mathbf{P}_\gamma$}
	\begin{aligned}
		\text{Minimize}&\quad J(y,u) + \frac\gamma2 \norm{\max\{0, y - y_b\}}_{L^2(\Omega)}^2 + \frac12 \norm{ u - \bar u}^2_{L^2(\Omega)}\\
		\text{with respect to}&\quad (y,u)\in H_0^1(\Omega)\times L^2(\Omega)\\
		\text{such that}& \quad y\in K ,\; \dual{\AA y-u}{v-y} \ge 0\quad\forall v\in K,\\
		\text{and}&\quad u \in \Uad.
	\end{aligned}
\end{equation}

Note that the term $\frac12 \norm{ u - \bar u}^2_{L^2(\Omega)}$ in the regularized objective functional 
is necessary to prove convergence if $(\bar y, \bar u)$ is not a strict local minimizer.

We proceed by proving that the minimizer $(\bar y, \bar u)$ can be approximated by local solutions of the regularized problem \eqref{eq:P:reg}.

\begin{lemma}
	\label{lem:existence_local_solution}
	There exists a sequence $\seq{\gamma_k}_{k \in \N}$
	with $\gamma_k \to \infty$,
	such that
	there exists a
	local solution $(y_k, u_k)$ of \eqref{eq:P:reg} with $\gamma = \gamma_k$
	for each $k \in \N$
	and $u_k \to \bar u$ in $L^2(\Omega)$.
	Thus, $y_k \to \bar y$ in $H_0^1(\Omega)$
	and
	$\xi_k := u_k - \AA y_k \to \bar u - \AA \bar y =: \bar\xi$
	in $H^{-1}(\Omega)$.
\end{lemma}
\begin{proof}
	We use a meanwhile classical localization argument.
	Let $\delta > 0$ denote the radius of optimality of $\bar u$.
	We introduce the auxiliary problems
	\begin{equation}
		\label{eq:P:reg_loc}
		\tag{$\mathbf{P}_{\gamma,\delta}$}
		\begin{aligned}
			\text{Minimize}&\quad J(y,u) + \frac\gamma2 \norm{\max\{0, y - y_b\}}_{L^2(\Omega)}^2 + \frac12 \norm{u - \bar u}^2_{L^2(\Omega)}\\
			\text{with respect to}&\quad (y,u)\in H_0^1(\Omega)\times L^2(\Omega)\\
			\text{such that}& \quad y\in K, \; \dual{\AA y-u}{v-y} \ge 0\quad\forall v\in K \\
			\text{and}& \quad u \in \Uad \cap B_\delta(\bar u)
			.
		\end{aligned}
	\end{equation}
	By the usual arguments, these problems possess global solutions.
	Let $\seq{\gamma_k}_{k \in \N}$ be an arbitrary sequence of positive numbers with $\gamma_k \to \infty$.
	We denote by $(y_k, u_k)$ a global solution of problem \eqref{eq:P:reg_loc} with $\gamma = \gamma_k$.

	Since the sequence $\seq{u_k}_{k \in \N}$ is bounded in $L^2(\Omega)$,
	we can extract a weakly convergent subsequence (denoted by the same symbol).
	We denote by $\tilde u$ the weak limit and
	by compact embedding, we have
	$y_k \to \tilde y := S(\tilde u)$ in $H_0^1(\Omega)$.
	Since $(\bar y, \bar u)$ is feasible for 
	\eqref{eq:P:reg_loc},
	we find
	\begin{equation}
		\label{eq:estimate_objectives}
		J(\bar y, \bar u) \ge J(y_k, u_k) + \frac{\gamma_k}2 \norm{\max\{0, y_k - y_b\}}_{L^2(\Omega)}^2
		+ \frac12 \norm{ u_k - \bar u}^2_{L^2(\Omega)}.
	\end{equation}
	From this inequality we infer that $\tilde y \le y_b$.
	Hence, $(\tilde y, \tilde u)$ is a feasible point for \eqref{eq:P:unreg}
	and by lower semicontinuity of $J$, we find
	\begin{equation*}
		J(\bar y, \bar u)
		\ge
		J(\tilde y, \tilde u) + \frac12 \limsup_{k \to \infty}\norm{ u_k - \bar u}^2_{L^2(\Omega)}
		\ge
		J(\tilde y, \tilde u) + \frac12 \norm{ \tilde u - \bar u}^2_{L^2(\Omega)}
		.
	\end{equation*}
	Since $\tilde u \in B_\delta(\bar u)$, we obtain
	\begin{equation*}
		(\tilde y, \tilde u) = (\bar y, \bar u)
		\quad\text{and}\quad
		\norm{u_k - \bar u}_{L^2(\Omega)} \to 0
		,
	\end{equation*}
	i.e.,
	$u_k \to \bar u$ in $L^2(\Omega)$.
	Hence, the constraint $u_k \in B_\delta(\bar u)$
	is not active for large $k$ and the result follows.
\end{proof}

The regularized problem 
\eqref{eq:P:reg}
is a
standard optimal control problem of the obstacle problem
with control constraints
and a differentiable objective function.
Thus, we obtain a primal optimality condition similar to \cref{lem:B_stat_first},
i.e.,
\begin{equation}
 \label{eq:B_stat_reg}
 0
 \le
 \dual{J_y(y_k, u_k)}{S'(u_k, h)}
 +
 \gamma \, \innerprod{\max\set{0,y_k - y_b}}{S'(u_k; h)}
 +
 \innerprod{J_u(y_k, u_k)}{h}
 +
 \innerprod{u_k - \bar u}{h}
\end{equation}
folds for all $h \in \TT_{\Uad}(u_k)$.
On the other hand,
local solutions satisfy a system of C-stationarity,
see \cite[Theorem~1.1]{Wachsmuth2014:2}
and (under higher regularity assumptions on the data) \cite[Propositions~3.5--3.8]{SchielaWachsmuth2013}.
This yields the following result.

\begin{lemma}
	\label{lem:optimality_condition_regularized}
	Let 
	$(y_k, u_k)$ be locally optimal for \eqref{eq:P:reg}.
	Then, there exist
	$\mu_k \in H^{-1}(\Omega)$, $\lambda_k \in L^2(\Omega)$ and $p_k \in H_0^1(\Omega)$
	such that the system
	\begin{subequations}
		\label{eq:strong_stationarity_low}
		\begin{align}
			\label{eq:strong_stationarity_low_1}
			\AA^\star p_k + J_y(y_k,u_k) + \gamma_k \max\{0, y_k - y_b\} + \mu_k &= 0 \quad\text{in } H^{-1}(\Omega), \\
			\label{eq:strong_stationarity_low_2}
			J_u(y_k,u_k) + (u_k - \bar u) + \lambda_k - p_k &= 0 \quad\text{in } L^2(\Omega), \\
			\label{eq:strong_stationarity_low_4}
			p_k & = 0 \quad\text{q.e.\ on } \Omega_{s,k}, \\
			\label{eq:strong_stationarity_low_5}
			\dual{\mu_k}{v}_{H_0^1(\Omega)}
			&= 0 \quad\forall v \in H_0^1(\Omega), v = 0 \text{ q.e.\ on } \Omega_{a,k},
			\\
			\label{eq:strong_stationarity_low_6}
			\dual{\mu_k}{\Phi p_k}_{H_0^1(\Omega)} &\ge 0
			\quad
			\forall \Phi \in W^{1,\infty}(\Omega)^+,
			\\
			\label{eq:strong_stationarity_low_7}
			\lambda_k &\in \NN_{\Uad}(u_k)
		\end{align}
	\end{subequations}
	is satisfied.
	Here,
	\begin{align*}
		\Omega_{a,k} &:= \set{y_k = y_a},
		&
		\Omega_{s,k} &:= \qsupp{\xi_k}
	\end{align*}
	are the active and strictly active set
	for the obstacle problem at $(y_k, u_k)$, respectively,
	and $\xi_k := u_k - \AA y_k$ is the corresponding multiplier.
	Note that both sets are defined up to sets of capacity zero.
\end{lemma}

\subsection{Boundedness of the multipliers}
\label{subsec:boundedness_mult}

From now on,
we will not only fix $(\bar y,\bar u)$ (with associated multiplier $\bar\xi$), but also
sequences
$\seq{\gamma_k}_{k \in \N}$
and
$\seq{(y_k, u_k)}_{k \in \N}$
as in \cref{lem:existence_local_solution}
and the corresponding sequences of multipliers
from \cref{lem:optimality_condition_regularized}.
Recall that \cref{lem:existence_local_solution}
already implies the convergence results for the primal quantities
$u_k,y_k,$ and $\xi_k$.
We check that this implies bounds on the dual variables,
in order to pass to the limit in the optimality system \eqref{eq:strong_stationarity_low} in \cref{lem:optimality_condition_regularized}.

For brevity, we introduce
the regularized counterpart to the multiplier $\nu$ for the state
constraint via
\begin{equation}
	\label{eq:def_nu_k}
	\nu_k
	:= \gamma_k \max\{0, y_k - y_b\}
	,
\end{equation}
as well as the set
on which the state constraint is violated or active,
i.e.,
\begin{equation*}
	\Omega_{b,k} := \set{y_k \ge y_b}.
\end{equation*}
Note that $\nu_k$ is an approximation of a Lagrange multiplier
for the pointwise state constraint $\bar y\le y_b$
in the unregularized problem \eqref{eq:P:unreg} with support contained in $\Omega_{b,k}$.
Our first goal is to bound $\nu_k$, $\mu_k,$ and $p_k$ in appropriate spaces.
To this end, we observe that the supports of $\nu_k$ and $\mu_k$ are uniformly separated.
\begin{lemma}
	\label{lem:separation_support}
	There exists a constant $\rho > 0$
	such that
	\begin{equation*}
		\dist(\Omega_{a,k}, \Omega_{b,k}) \ge \rho
	\end{equation*}
	holds for all $k$.
\end{lemma}
\begin{proof}
	By \cref{asm:big_list}~\ref{item:asm_domain} the controls $u_k$ are bounded in $L^2(\Omega)$, and the associated states $y_k$
	are bounded in $H^2(\Omega)$ due to the mapping properties of $S$.
	Hence, their Hölder-norm is uniformly bounded
	and the result follows from $y_a \le y_b - \tau$.
\end{proof}
As a consequence, we obtain the following auxiliary result:
\begin{lemma}
 \label{lem:sets}
 There exist open sets $\hat \Omega_a\supset \Omega_a$ and $\hat \Omega_b\supset \Omega_b$
 such that $\Omega_{a,k}\subset \hat \Omega_a$, $\Omega_{b,k}\subset \hat \Omega_b$
 for all $k$ sufficiently large as well as $\rho>0$ 
 such that
 \begin{equation*}
  \dist(\hat \Omega_{a}, \hat \Omega_{b}) > \rho.
 \end{equation*}
\end{lemma}
\begin{proof}
 We define
 \begin{align*}
  \hat \Omega_a &:=\set[\Big]{\bar y < y_a + \frac{\tau}{4}},
  &
  \hat \Omega_b &:=\set[\Big]{\bar y > y_b - \frac{\tau}{4}}.
 \end{align*}
 For $x\in \Omega_{a,k}$, we observe that
 $\bar y(x)=\bar y(x)-y_k(x)+y_k(x)\le \norm{\bar y-y_k}_{L^\infty(\Omega)}+y_a(x)$.
 For $k$ large enough, uniform convergence of $y_k$ towards $\bar y$ yields $x\in \hat\Omega_{a}$.
 The set $\hat\Omega_b$ can be treated analogously.
 From $y_b-\frac{\tau}{4}-y_a-\frac{\tau}{4}< \tau-\frac{\tau}{2}= \frac{\tau}{2}$
 and the Hölder continuity of $\bar y$ the result follows.
\end{proof}
The boundedness
of the multiplier approximations $\nu_k$
is a simple consequence of the Slater point property.
\begin{lemma}
	\label{lem:boundedness_nu}
	There exists $C > 0$ such that $\norm{\nu_k}_{L^1(\Omega)} \le C$.
\end{lemma}
\begin{proof}
We start with the B-stationarity \eqref{eq:B_stat_reg} with $h = \hat u - u_k$, i.e.,
\begin{align*}
 0&\le
 \dual{J_y(y_k,u_k)}{S'(u_k; \hat u-u_k)}
 +
 \gamma_k \, \innerprod{\max\{0,y_k-y_b\}}{S'(u_k; \hat u-u_k)}
 \\&\qquad+
 \innerprod{J_u(y_k,u_k)}{\hat u-u_k}
 +
 \innerprod{u_k - \bar u}{\hat u - u_k}
 .
\end{align*}
Due to the convergence properties of $y_k$ and $u_k$,
the first, third and fourth addend can be bounded by a constant.
Thus,
\begin{equation*}
 \gamma_k \, \innerprod{\max\{0,y_k-y_b\}}{S'(u_k; \hat u-u_k)}
 \ge
 -C.
\end{equation*}
Due to the convexity of the solution operator $S$, \cref{cor:conv_direc_deriv} can be used
to obtain a linearized Slater condition for the local solutions $u_k$ of \eqref{eq:P:reg}
from the Slater point $\hat u$. Indeed, for all $k>0$ we have
\[
 y_k + S'(u_k; \hat u-u_k)\le S(\hat u)\le y_b-\tau.
\]
Combining the last two inequalities yields
\begin{equation*}
 \innerprod{\nu_k}{y_b - y_k - \tau}
 =
 \gamma_k \, \innerprod{\max\{0,y_k-y_b\}}{y_b - y_k - \tau}
 \ge
 -C.
\end{equation*}
Since $\innerprod{\nu_k}{y_b - y_k} \le 0$
by definition of $\nu_k$,
we obtain
$\norm{\nu_k}_{L^1(\Omega)} \le C \tau^{-1}$.
\end{proof}

Next, we show the boundedness of the adjoint state $p$ and of the multiplier $\mu$.
\begin{lemma}
 \label{lem:bound_p_mu}
 For every $q \in \parens[\big]{1,n/(n-1)}$, there exists $C > 0$ such that
 \begin{equation*}
  \norm{p_k}_{W_0^{1,q}(\Omega)}
  +
  \norm{p_k}_{H^{1}(\hat\Omega_a)}
  +
  \norm{\mu_k}_{H^{-1}(\Omega)}
  \le
  C,
 \end{equation*}
 where $\hat\Omega_a$ is defined in \cref{lem:sets}.
\end{lemma}
\begin{proof}
We split the adjoints $p_k$ into the sum of
$p_k^y,p_k^\mu,p_k^\nu \in H_0^1(\Omega)$,
defined via
\begin{align*}
		\AA^\star p_k^y + J_y(y_k, u_k) &= 0 \quad\text{in } H^{-1}(\Omega), \\
		\AA^\star p_k^\mu + \mu_k &= 0 \quad\text{in } H^{-1}(\Omega), \\
		\AA^\star p_k^\nu + \gamma_k \max\{0, y_k - y_b\} &= 0 \quad\text{in } H^{-1}(\Omega).
\end{align*}
Due to $(y_k, u_k) \to (\bar y, \bar u)$,
the term $J_y(y_k, u_k)$ is bounded in $H^{-1}(\Omega)$,
see \cref{asm:big_list}~\ref{item:asm_objective}.
This implies the boundedness of $p^y_k$, i.e.,
\begin{equation*}
 \norm{p^y_k}_{H_0^1(\Omega)} \le C.
\end{equation*}
Next, we are going to bound $p^\nu_k$.
To this end, let $\varphi \in C_c^\infty(\Omega)$
be given, such that $\varphi = 1$ on $\hat\Omega_a$
and $\varphi = 0$ on $\hat\Omega_b$.
Since the equation for $p^\nu_k$ can be understood in the very weak sense,
see \eqref{eq:adjveryweak2},
we can apply
\cref{thm:existenceadjoint,lem:bound_adjoint}
in combination with \cref{lem:boundedness_nu}
to obtain
\begin{equation*}
 \norm{p^\nu_k}_{W_0^{1,q}(\Omega)}
 +
 \norm{\varphi p^\nu_k}_{H_0^1(\Omega)}
 \le
 C.
\end{equation*}
To obtain a uniform bound for $p_k^\mu$, we write
\begin{equation*}
	\gamma_1 \norm{p_k^\mu}_{H_0^1(\Omega)}^2
	\le
	\dual{\AA^\star p_k^\mu}{p_k^\mu}
	=
	-\dual{\mu_k}{p_k^\mu}
	=
	\dual{\mu_k}{p_k^\nu}
	+\dual{\mu_k}{p_k^y}
	-\dual{\mu_k}{p_k}
	.
\end{equation*}
In order to bound the first term,
we use
\begin{equation*}
	\dual{\mu_k}{(1 - \varphi) v} = 0
	\qquad\forall v \in H_0^1(\Omega)
\end{equation*}
due to \eqref{eq:strong_stationarity_low_5}.
For the third term,
we can apply \eqref{eq:strong_stationarity_low_6} with $\Phi = 1$
and obtain $-\dual{\mu_k}{p_k} \le 0$.
Now,
the above inequality yields
\begin{align*}
	\gamma_1 \norm{p_k^\mu}_{H_0^1(\Omega)}^2
	&\le
	\dual{\mu_k}{\varphi p_k^\nu}
	+\dual{\mu_k}{p_k^y}
	\\
	&\le
	C \norm{\mu_k}_{H^{-1}(\Omega)}\norm{\varphi p_k^\nu}_{H_0^1(\Omega)}
	+ C \norm{\mu_k}_{H^{-1}(\Omega)} \norm{p_k^y}_{H_0^1(\Omega)}.
\end{align*}
Together with
\begin{equation*}
	C^{-1} \norm{p_k^\mu}_{H_0^1(\Omega)}
	\le
	\norm{\mu_k}_{H^{-1}(\Omega)}
	\le
	C \norm{p_k^\mu}_{H_0^1(\Omega)}
\end{equation*}
which follows from the coercivity of $\AA\adjoint$,
we obtain the claim.
\end{proof}

\subsection{Passage to the limit in the optimality system}
\label{subsec:passage_limit}

From the boundedness results in \cref{lem:boundedness_nu} and \cref{lem:bound_p_mu} we conclude that there exist weakly convergent subsequences, denoted by the same index $k$, satisfying
		\begin{align*}
	p_k&\weakly p\quad\text{in } W_0^{1,q}(\Omega), &
	\nu_k&\weaklystar \nu\quad\text{in } \MM(\Omega), &
	\mu_k&\weakly \mu\quad\text{in } H^{-1}(\Omega).
	\end{align*}
	In the following steps, we will prove that the limits satisfy the optimality system of \cref{thm:c_stationarity}.
To this end, we recall the strong convergences
	\begin{align*}
	u_k&\to \bar u\quad\text{in } L^{2}(\Omega), &
	y_k&\to \bar y\quad\text{in } H_0^1(\Omega)
	\end{align*}
	from \cref{lem:existence_local_solution}, and note that the strong convergence
	 	\begin{equation*}
		\lambda_k\to \lambda\quad\text{in } L^2(\Omega)
		\end{equation*}
	with $\lambda\in \NN_{\Uad}(\bar u)$ is then a simple consequence
	of the gradient equation \eqref{eq:strong_stationarity_low_2} in \cref{lem:optimality_condition_regularized}
	and the closedness of the graph of the normal cone.
	This proves \eqref{eq:c_stat_2} and \eqref{eq:c_stat_6}.
	Also, \eqref{eq:c_stat_1} is immediately clear.
	It remains to prove
		the complementarity condition for $\nu$, as well as the properties \eqref{eq:c_stat_3}-\eqref{eq:c_stat_45}.
 First, let us show that the weak limit $\nu$ fulfills the complementarity condition \eqref{eq:c_stat_5} for the unregularized problem \eqref{eq:P:unreg}.
\begin{lemma}
	\label{lem:complementarity_nu}
	The weak limit $\nu$ fulfills $\nu\ge 0$ as well as $\supp(\nu) \subset \Omega_b$, see \eqref{eq:c_stat_5}.
\end{lemma}
\begin{proof}
Nonnegativity of $\nu$ is an immediate consequence of $\nu_k\ge 0$ for all $k$.
Moreover, from \eqref{eq:estimate_objectives},
	we observe
	\begin{equation*}
		\int_\Omega \nu_k \, (y_k - y_b) \, \d x
		=
		\gamma_k \int_\Omega \max\{0, y_k - y_b\}^2 \, \d x
		\to 0.
	\end{equation*}
	The mapping properties of $S$
	guarantee
		$y_k \to \bar y$
	in $C_0(\Omega)$, and hence
	\begin{equation*}
		\dual{\nu}{\bar y - y_b}_{C_0(\Omega)}
		=
		0.
	\end{equation*}
Feasibility of $\bar y$, i.e.\ $\bar y-y_b\le 0$ concludes the proof.
\end{proof}
The conditions \eqref{eq:c_stat_3} and \eqref{eq:c_stat_4} on $p$ and $\mu$ follow from results in \cite{Wachsmuth2014:2}:
\begin{lemma}
	\label{lem:condition_on_p}
	The weak limit $p$ of $\seq{p_k}$ satisfies $p= 0$ q.e.\ on $\qsupp(\bar\xi)$, see \eqref{eq:c_stat_3}.
\end{lemma}
\begin{proof}
	This follows from \eqref{eq:strong_stationarity_low_4} via \cite[Lemma~4.2]{Wachsmuth2014:2}.
\end{proof}
 
\begin{lemma}
	\label{lem:condition_on_mu}
	The weak limit $\mu$ of $\seq{\mu_k}$ satisfies
	\begin{equation*}
		\mu\in\set{v\in H_0^1(\Omega)\given v=0 \text{ q.e. on } \set{\bar y = y_a} }\anni,
	\end{equation*}
	see  \eqref{eq:c_stat_4}.
\end{lemma}
\begin{proof}
	This follows from \eqref{eq:strong_stationarity_low_5} via \cite[Lemma~4.3]{Wachsmuth2014:2}.
\end{proof}
Finally, we prove \eqref{eq:c_stat_45}.
\begin{lemma}
	\label{lem:sign_p_mu}
The weak limits $p$ and  $\mu$ fulfill
$\dual{\mu}{\Phi p}_{H_0^1(\Omega)} \ge 0$ for all $\Phi \in W^{1,\infty}(\Omega)^+$ that satisfy $\Phi_{|_{\Omega \setminus \hat\Omega_a}}=0,$ see \eqref{eq:c_stat_45}.
\end{lemma}
\begin{proof}
	From \eqref{eq:strong_stationarity_low_6} in the C-stationarity system of \cref{lem:optimality_condition_regularized},
	we know
	\begin{equation*}
		\dual{\mu_k}{p_k \Phi}
		\ge
		0
		\qquad\forall \Phi \in W^{1,\infty}(\Omega)^+
		.
	\end{equation*}
	Using the separation of the adjoint state into $p_k^y,p_k^\mu,$ and $p_k^\nu$, we therefore observe
	\begin{equation}\label{eq:aux}
0\le		\dual{\mu_k}{p_k \Phi}
		=
		\dual{\mu_k}{(p_k^y + p_k^\mu + p_k^\nu) \Phi}
		=
		\dual{\mu_k}{p_k^y \Phi}
		+
		\dual{-\AA^\star p_k^\mu}{p_k^\mu \Phi}
		+
		\dual{\mu_k}{p_k^\nu \Phi}
		.
	\end{equation}
	For the first term on the right-hand-side of \eqref{eq:aux}, we note that $p_k^y$ converges strongly in $H_0^1(\Omega)$ due to the mapping properties of $T^\star$, cf. \cref{thm:existenceh1} which is applicable to the adjoint equation. This yields $\dual{\mu_k}{p_k^y\Phi}\to\dual{\mu}{p^y\Phi}$. The arguments in
	\cite[Proof of Lemma~4.5, (4.2)]{Wachsmuth2014:2} applied to the second term yield 
	\begin{equation*}\limsup_{k\to\infty}\dual{-\AA^\star p_k^\mu}{p_k^\mu \Phi}\le \dual{ -\AA^\star p^\mu}{p^\mu \Phi}.\end{equation*}
	Finally, for the third term $\dual{\mu_k}{p_k^\nu \Phi}$ we apply the separation of sets from \cref{lem:sets}, and point out that $\Phi p\in H_0^1(\Omega)$.
	Note that the supports of $\nu$ and all $\nu_k$ are contained in $\hat\Omega_b$. 
	We apply \cref{lem:bound_adjoint} with $U=\hat\Omega_b$ and $V$ an open set containing $\hat\Omega_b$ with positive distance to $\hat\Omega_a$, $\varphi\ge 0$, $\varphi=1$ on $\hat\Omega_a$, and $\varphi=0$ on $\hat\Omega_b$. Hence, $\mu_k = \varphi \mu_k$ converges weakly towards $\mu = \varphi \mu$ in $H^{-1}(\Omega)$ and $\varphi p_k^\nu$ converges strongly to $\varphi p^\nu$ in $H_0^1(\Omega)$.  Thus we obtain
	\begin{equation*}
		\dual{\mu_k}{p_k^\nu \Phi}
		=
		\dual{\mu_k}{\varphi p_k^\nu \Phi}
		\to
		\dual{\mu}{\varphi p^\nu \Phi}
		=
		\dual{\mu}{p^\nu \Phi}.
	\end{equation*}
	Collecting all arguments yields the assertion.
\end{proof}

The proof of \cref{thm:c_stationarity} is complete.

\section{Strong stationarity without control constraints}
\label{sec:strong_stat}
In this section, we consider the problem \eqref{eq:P:unreg}
in the case without control constraints, i.e., $\Uad = L^2(\Omega)$.
We will show that in this case local solutions are strongly stationary.
As a byproduct,
we will obtain a characterization of the normal cone of $\Ustate$.

In the following,
we will follow the approach by \cite[Théorème~4.3]{Mignot1976}
and show the system of strong stationarity
by employing \cref{thm:B_stat_second}.
On the one hand,
this has the advantage
of showing the equivalency of B-stationarity and strong stationarity.
On the other hand,
it enables us to derive the announced characterization
of the normal cone of $\Ustate$.

Before we dive into the proofs,
let us state the system of strong stationarity.
Note that we state the system
without assuming $\Uad = L^2(\Omega)$.
However, we only show that it is a necessary optimality condition
in case $\Uad = L^2(\Omega)$.
\begin{definition}
	\label{def:strong_stationarity}
	Let an admissible control $\bar u \in \Ueff$
	be given.
	We denote by $\bar y = S(\bar u)$
	and $\bar\xi = \bar u - \AA \bar y$
	the associated state and multiplier.
	We say that $\bar u$ is strongly stationary
	if there exist
	$p \in W_0^{1,q}(\Omega)$,
	$\mu \in H^{-1}(\Omega)$,
	$\nu \in \MM(\Omega)^+$,
	$\lambda \in L^2(\Omega)$
	such that
	$p \in H^1(\hat\Omega_a)$
	for some open $\hat\Omega_a \supset \set{\bar y = y_a}$
	and
	such that the system
	\begin{subequations}
		\label{eq:strong_stat}
		\begin{align}
			\label{eq:strong_stat_1}
			\AA\adjoint p + J_y(\bar y,\bar u) + \nu + \mu &= 0
			\\
			\label{eq:strong_stat_2}
			J_u(\bar y,\bar u) + \lambda - p &= 0
			\\
			\label{eq:strong_stat_3}
			p &= 0 \text{ q.e.\ on } \qsupp(\bar \xi), \quad
			p \le 0 \text{ q.e.\ on } \set{\bar y = y_a}
			\\
			\label{eq:strong_stat_4}
			\mu &\in \KK(\bar u)\polar
			\\
			\label{eq:strong_stat_5}
			\supp(\nu) &\subset \Omega_b
			\\
			\label{eq:strong_stat_6}
			\lambda &\in \NN_{\Uad}(\bar u)
		\end{align}
	\end{subequations}
	is satisfied.
	Here,
	$\Omega_s := \qsupp(\bar\xi)$,
	$\Omega_a := \set{\bar y = y_a}$,
	$\Omega_b = \set{\bar y = y_b}$,
	and
	the adjoint equation is to be understood in the very weak sense,
	see \eqref{eq:adjveryweak2}.
\end{definition}
Note that \eqref{eq:strong_stat_6} implies $\lambda = 0$ in case $\Uad = L^2(\Omega)$.

In the sequel of this section,
we will use two smooth test functions.
These functions have the properties
\begin{subequations}
	\label{eq:smotoh_fuct}
	\begin{align}
		\label{eq:smotoh_fuct_1}
		\psi &\in C_c^\infty(\Omega),
		&
		0 &\le \psi \le 1 \text{ in } \Omega,
		&
		\psi &= 0 \text{ in nbhd. of } \Omega_a,
		&
		\psi &= 1 \text{ in nbhd. of } \Omega_b,
		\\
		\varphi &\in C^\infty(\R^n)
		&
		0 &\le \varphi \le 1 \text{ in } \Omega,
		&
		\varphi &= 1 \text{ in } \set{\psi < 1},
		&
		\varphi &= 0 \text{ in nbhd. of } \Omega_b.
		\label{eq:smotoh_fuct_2}
	\end{align}
\end{subequations}
We fix $\varphi$ and $\psi$ throughout this section.
Note that such a choice of $\varphi$ and $\psi$
is possible
since $\Omega_a$ and $\Omega_b$ have a positive
distance
and since $\Omega_b$ has a positive distance to the boundary.

Further,
we argue that
the regularity $p \in H^1(\hat\Omega_a)$
is enough to write down the condition \eqref{eq:strong_stat_3}.
Indeed, it
implies that $\varphi p \in H_0^1(\Omega)$
has a quasi-continuous representative.
Since $\varphi = 1$ in a neighborhood of $\Omega_a$,
$p$ is quasi-continuous on $\Omega_a$
and therefore it makes sense to state \eqref{eq:strong_stat_3}.
In the case that the adjoint state has additionally the regularity $p \in H_0^1(\Omega)$,
one can formulate \eqref{eq:strong_stat_3}
as $p \in -\KK(\bar u)$.

Now, we start with the B-stationarity system
\cref{thm:B_stat_second}.
In order to satisfy \eqref{eq:strong_stat_2} and \eqref{eq:strong_stat_6},
we set
$p = J_u(\bar y,\bar u)$ in case $\Uad = L^2(\Omega)$.
The differentiability properties of $J$
only yield the low regularity $p \in L^2(\Omega)$.
In the next two results,
we show that $p$ enjoys some increased regularity
and that the first-order condition
\eqref{eq:B_stat_second}
can be extended to a larger test spaces.
\begin{lemma}
	\label{lem:p_in_W1q}
	We assume $\Uad = L^2(\Omega)$.
	Let $\bar u \in \Ueff$ satisfy \eqref{eq:B_stat_second}.
	Then, the function $p := J_u(\bar y,\bar u) \in L^2(\Omega)$
	satisfies $p \in W^{1,q}(\Omega)$
	for all $q < n / (n - 1)$.
	Moreover,
	\begin{equation}
		\label{eq:B_stat_W1p}
		\dual{J_y(\bar y,\bar u)}{S'(\bar u; h)} + \dual{h}{p} \ge 0
		\qquad\forall
		h \in W^{-1,q'}(\Omega),
		S'(\bar u; h) \le 0 \text{ on } \Omega_b,
	\end{equation}
	where
	$\Omega_b := \set{\bar y = y_b}$ with $\bar y=S(\bar u)$.
\end{lemma}
\begin{proof}
	Let $q \in (1, n / (n-1))$ be given,
	i.e., $q' \in (n, \infty)$.
	For all $h \in L^2(\Omega)$,
	we have $z_h := S'(\bar u; h)$
	satisfies
	$\AA z = h - \upsilon_h$
	with $\upsilon_h \in \KK(\bar u)\polar$
	and $\norm{\upsilon_h}_{H^{-1}(\Omega)} \le \norm{h}_{H^{-1}(\Omega)}$
	Now, from
	\cref{lem:better_differentiability}~\ref{item:better_differentiability_1}
	we get the estimate
	\begin{equation*}
		\norm{ S'(\bar u; h) }_{L^\infty(\Omega_b)}
		=
		\norm{ S'(\bar u; h) - S'(\bar u; 0) }_{L^\infty(\Omega_b)}
		\le
		C \norm{h}_{W^{-1,q'}(\Omega)},
	\end{equation*}
	where $C > 0$ is independent of $h$.
	We set $c := \zeta / C$.
	Then, for all $h \in L^2(\Omega)$
	with $\norm{h}_{W^{-1,q'}(\Omega)} \le c$
	we have
	\begin{align*}
		S'(\bar u; h + (\hat u - \bar u))
		&\le
		S'(\bar u; h)
		+
		S'(\bar u; \hat u - \bar u)
		\\&
		\le
		S'(\bar u; h) + S(\hat u) - S(\bar u)
		\le
		\zeta + (y_b - \zeta) - y_b
		=
		0
	\end{align*}
	on $\Omega_b$, see \cref{cor:conv_direc_deriv}.
	Therefore, we can use
	$h + (\hat u - \bar u)$
	as a test function in
	\cref{thm:B_stat_second}
	and obtain
	\begin{equation*}
		\dual{J_y(\bar y,\bar u)}{S'(\bar u; h + (\hat u - \bar u))}
		+
		\innerprod{J_u(\bar y,\bar u)}{h + (\hat u - \bar u)}
		\ge
		0
		.
	\end{equation*}
	Thus, there is $K > 0$
	such that
	\begin{equation*}
		\innerprod{J_u(\bar y,\bar u)}{h}
		\ge
		-\dual{J_y(\bar y,\bar u)}{S'(\bar u; h + (\hat u - \bar u))}
		-\innerprod{J_u(\bar y,\bar u)}{\hat u - \bar u}
		\ge
		-K
	\end{equation*}
	holds for all $h \in L^2(\Omega)$
	with $\norm{h}_{W^{-1,q}(\Omega)} \le c$.
	Since we can replace $h$ by $-h$, we infer
	\begin{equation*}
		\abs{\innerprod{J_u(\bar y,\bar u)}{h}} \le K
		\qquad\forall h \in L^2(\Omega), 
		\norm{h}_{W^{-1,q}(\Omega)} \le C.
	\end{equation*}
	By scaling we get
	\begin{equation*}
		\abs{\innerprod{J_u(\bar y,\bar u)}{h}} \le \frac{K}{C} \norm{h}_{W^{-1,q}(\Omega)}
		\qquad\forall h \in L^2(\Omega). 
	\end{equation*}
	This and the density of $L^2(\Omega)$ in $W^{-1,q}(\Omega)$,
	imply
	$p = J_u(\bar y,\bar u) \in W^{1,q}_0(\Omega)$.

	It remains to show \eqref{eq:B_stat_W1p}.
	Let $h \in W^{-1,q'}(\Omega)$
	with $S'(\bar u; h) \le 0$ on $\Omega_b$ be given.
	Then, there is a sequence $\seq{h_k}_{k \in \N} \subset L^2(\Omega)$
	with $h_k \to h$ in $W^{-1,q'}(\Omega)$.
	Then,
	\cref{lem:better_differentiability}~\ref{item:better_differentiability_1}
	implies
	\begin{equation*}
		r_k
		:=
		\norm{S'(\bar u; h_k) - S'(\bar u; h)}_{L^\infty(\Omega_b)}
		\to
		0
		.
	\end{equation*}
	Thus,
	\begin{equation*}
		S'(\bar u; h_k + r_k \zeta^{-1} \, (\hat u - \bar u))
		\le
		0
		\qquad\text{a.e.\ on } \Omega_b.
	\end{equation*}
	Hence,
	\begin{equation*}
		\dual{J_y(\bar y,\bar u)}{S'(\bar u; h_k + r_k \zeta^{-1} \, (\hat u - \bar u))}
		+
		\innerprod{J_u(\bar y,\bar u)}{h_k + r_k \zeta^{-1} \, (\hat u - \bar u)}
		\ge
		0
		.
	\end{equation*}
	Due to $J_u(\bar y,\bar u) = p \in W_0^{1,q}(\Omega)$,
	we can pass to the limit $k \to \infty$ and obtain
	\eqref{eq:B_stat_W1p}.
\end{proof}

Using similar arguments,
we get that $p$ has $H^1$-regularity
if we stay away from the active set $\Omega_b$.

\begin{lemma}
	\label{lem:p_in_H1}
	We assume $\Uad = L^2(\Omega)$.
	Let $\bar u \in \Ueff$ satisfy \eqref{eq:B_stat_second}
	and
	$p := J_u(\bar y,\bar u) \in L^2(\Omega)$.
	We have
	$\varphi p \in H_0^1(\Omega)$
	and
	\begin{equation}
		\label{eq:B_stat_H1}
		\dual{J_y(\bar y,\bar u)}{S'(\bar u; \varphi h)} + \dual{h}{\varphi p} \ge 0
		\qquad\forall
		h \in H^{-1}(\Omega),
		S'(\bar u; \varphi h) \le 0 \text{ on } \Omega_b.
	\end{equation}
\end{lemma}
\begin{proof}
	The proof is very similar to the proof of \cref{lem:p_in_W1q}.
	We mainly have to replace the regularity result
	\cref{lem:better_differentiability}~\ref{item:better_differentiability_1}
	by
	\cref{lem:better_differentiability}~\ref{item:better_differentiability_3}.
	This yields
	\begin{equation*}
		\norm{
			S'(\bar u; \varphi h)
		}_{L^\infty(\Omega_b)}
		\le
		C \norm{h}_{H^{-1}(\Omega)}
	\end{equation*}
	for all $h \in L^2(\Omega)$,
	where $C > 0$ is independent of $h$.
	Now we can argue along the lines of the proof of \cref{lem:p_in_W1q}.
\end{proof}

Using this extended stationarity condition,
we can show the sign conditions on $p$.
\begin{lemma}
	\label{lem:sign_conditions_p}
	We assume $\Uad = L^2(\Omega)$.
	Let $\bar u \in \Ueff$ satisfy \eqref{eq:B_stat_second}
	and
	$p := J_u(\bar y,\bar u) \in L^2(\Omega)$.
	Then, $p \in H^1(\hat\Omega_a)$
	for some open $\hat\Omega_a \supset \Omega_a := \set{S(\bar u) = y_a}$
	and
	\eqref{eq:strong_stat_3} holds.
\end{lemma}
\begin{proof}
	We choose $\hat\Omega_a$ with a positive distance to
	$\Omega_b := \set{S(\bar u) = y_b}$.
	Then, there exist $\varphi \in C_c^\infty(\Omega)$
	such that
	the support of $\varphi$ does not intersect $\Omega_b$,
	$\varphi \ge 0$
	and $\varphi = 1$ on $\hat\Omega_a$.
	Now, \cref{lem:p_in_H1}
	implies $p \in H^1(\hat\Omega_a)$.

	Next,
	let $h \in \KK(\bar u)\polar$
	be arbitrary.
	Since $\varphi v \in \KK(\bar u)$
	for all $v \in \KK(\bar u)$,
	we have
	$\varphi h \in \KK(\bar u)\polar$ as well.
	Thus,
	$S'(\bar u; \varphi h) = 0$,
	see \cref{thm:direc_diff}.
	Thus, \eqref{eq:B_stat_H1} implies
	\begin{equation*}
		\dual{h}{\varphi p} \ge 0
		\qquad\forall h \in \KK(\bar u)\polar.
	\end{equation*}
	Hence, $\varphi p \in -\KK(\bar u)$.
	Since $\varphi = 1$ in a neighborhood of $\Omega_a$,
	this shows \eqref{eq:strong_stat_3}.
\end{proof}

It remains to verify the adjoint equation
and the sign conditions on $\mu$ and $\nu$.
First,
we consider the adjoint equation in a neighborhood of $\Omega_b$.
\begin{lemma}
	\label{lem:adjoint_PDE_nu}
	We assume $\Uad = L^2(\Omega)$.
	Let $\bar u \in \Ueff$ satisfy \eqref{eq:B_stat_second}
	and
	$p := J_u(\bar y,\bar u) \in W_0^{1,q}(\Omega)$
	with arbitrary $q \in (2 n / (n + 2), n/(n-1))$.
	Then, there is $\nu \in \MM(\Omega)$
	such that \eqref{eq:strong_stat_5}
	and
	\begin{equation}
		\label{eq:adjoint_PDE_nu}
		\dual{\AA (\psi z)}{p}
		+
		\dual{J_y(\bar y,\bar u)}{\psi z}
		+
		\int_\Omega z \, \d\nu
		=
		0
		\qquad\forall z \in Z,
	\end{equation}
	where $Z$ is defined in \eqref{eq:def_Z}.
\end{lemma}
\begin{proof}
	We have
	$1/q \in (1 - 1/n, 1/2 + 1/n)$
	and
	$1/q' \in (1/2 - 1/n, 1/n)$.
	Thus, \cref{lem:smooth_multiplier}
	implies $\AA (\psi z) \in W^{-1,q'}(\Omega)$,
	hence the first term in \eqref{eq:adjoint_PDE_nu}
	is well defined.

	Due to the properties of $S'(\bar u; \cdot)$,
	we have
	$S'(\bar u; \AA(\psi z)) = \psi z$
	for all $z \in Z$.
	If, additionally, $z \le 0$ on $\Omega_b$,
	we have by \eqref{eq:B_stat_W1p}
	\begin{equation*}
		\dual{J_y(\bar y,\bar u)}{\psi z} + \dual{\AA(\psi z)}{p} \ge 0
		.
	\end{equation*}
	Hence,
	the left-hand side defines a negative
	functional w.r.t.
	$z \in Z$.
	Moreover, if $M := \norm{z}_{C_0(\Omega)}$,
	we have $\varphi \, (z + M) \ge 0$ on $\Omega_b$,
	thus
	\begin{equation*}
		\dual{J_y(\bar y,\bar u)}{\psi \, (z + M)} + \dual{\AA(\psi \, (z + M))}{p} \ge 0
		,
	\end{equation*}
	i.e.,
	\begin{equation*}
		\dual{J_y(\bar y,\bar u)}{\psi z} + \dual{\AA(\psi z)}{p}
		\ge
		-\dual{J_y(\bar y,\bar u)}{M \psi} - \dual{\AA(M \psi)}{p}
		=:
		-C M
		.
	\end{equation*}
	Similarly, by considering $\varphi \, (z - M)$,
	one can show that the left-hand side
	is bounded from above by $C M$.
	Since $C_c^\infty(\Omega)^+ \subset Z$ is dense in $C_0(\Omega)^+$,
	the first two addends in \eqref{eq:adjoint_PDE_nu}
	define a negative Borel measure $-\nu \in \MM(\Omega)$.
	This shows \eqref{eq:adjoint_PDE_nu}.
	Moreover,
	by considering $z \in C_c^\infty(\Omega)$
	with $z = 0$ on $\Omega_b$ is arbitrary,
	we get $\int_\Omega z \, \d\nu = 0$.
	Hence, \eqref{eq:strong_stat_5}
	follows.
\end{proof}
Next,
we argue in the neighborhood of $\Omega_a$.
To this end, we use the test function $\varphi$
from \eqref{eq:smotoh_fuct_2}.
\begin{lemma}
	\label{lem:adjoint_PDE_mu}
	We assume $\Uad = L^2(\Omega)$.
	Let $\bar u \in \Ueff$ satisfy \eqref{eq:B_stat_second}
	and
	$p := J_u(\bar y\bar u)$.
	We define $\mu \in H^{-1}(\Omega)$
	via
	\begin{equation}
		\label{eq:adjoint_PDE_mu}
		\dual{\mu}{v}
		:=
		-\bracks[\big]{
			\dual{J_y(\bar y,\bar u)}{(1-\psi) v}
			+
			\dual{\AA ((1-\psi) v)}{p}
		}
		\qquad
		\forall v \in H_0^1(\Omega)
		.
	\end{equation}
	Then,
	\eqref{eq:strong_stat_4}
	is satisfied.
\end{lemma}
\begin{proof}
	Since $\varphi p \in H_0^1(\Omega)$
	by \cref{lem:p_in_H1},
	we have $p \in H^1(\set{\psi < 1})$.
	Thus,
	the definition \eqref{eq:adjoint_PDE_mu}
	implies the regularity $\mu \in H^{-1}(\Omega)$.

	In order to check \eqref{eq:strong_stat_4},
	we take an arbitrary $v \in \KK(\bar u)$.
	Then, $(1-\psi) v \in \KK(\bar u)$ as well
	and, consequently,
	$S'(\bar u; h) = (1-\psi) v$
	for $h = \AA ( (1-\psi) v)$.
	Note that $h = \varphi h$ due to the construction of $\varphi$ and $\psi$,
	i.e., $S'(\bar u; \varphi h) = (1-\psi) v$ as well.
	Since $(1-\psi) v = 0$ on $\Omega_b$,
	we can use $h$ in \eqref{eq:B_stat_H1}
	and obtain
	\begin{align*}
		-\dual{\mu}{v}
		&=
		\dual{J_y(\bar y,\bar u)}{ (1-\psi) v }
		+
		\dual{\AA ( (1-\psi) v)}{p}
		\\
		&=
		\dual{J_y(\bar y,\bar u)}{ (1-\psi) v }
		+
		\dual{\AA ( (1-\psi) v)}{\varphi p}
		\ge
		0.
	\end{align*}
	Since $v \in \KK(\bar u)$ was arbitrary,
	this shows $\mu \in \KK(\bar u)\polar$.
\end{proof}
By collection of the results of
\cref{lem:p_in_W1q,lem:p_in_H1,lem:sign_conditions_p,lem:adjoint_PDE_nu,lem:adjoint_PDE_mu},
we can show that the system of strong stationarity
is equivalent to the B-stationarity from \cref{thm:B_stat_second}.
\begin{theorem}
	\label{thm:strong}
	We assume $\Uad = L^2(\Omega)$ and let $\bar u \in \Ustate$ be given.
	Then,
	$\bar u$ is strongly stationary
	if and only if
	the B-stationarity \eqref{eq:B_stat_second}
	is satisfied.
\end{theorem}
\begin{proof}
	``$\Leftarrow$'':
	By using the results from the previous lemmas,
	it remains to show that the adjoint PDE \eqref{eq:strong_stat_1}
	is satisfied.
	To this end, let $z \in Z$ be arbitrary.
	Using \eqref{eq:adjoint_PDE_nu} and \eqref{eq:adjoint_PDE_mu},
	we have
	\begin{align*}
		\dual{J_y(\bar y,\bar u)}{z}
		+
		\dual{\AA z}{p}
		&=
		\dual{J_y(\bar y,\bar u)}{(1-\psi) z}
		+
		\dual{\AA ((1-\psi) z)}{p}
		\\&\qquad
		+
		\dual{J_y(\bar y,\bar u)}{\psi z}
		+
		\dual{\AA (\psi z)}{p}
		\\
		&=
		-\dual{\mu}{\nu}
		-\int_\Omega z \, \d\nu
		.
	\end{align*}
	Hence, the adjoint PDE is satisfied.

	``$\Rightarrow$'':
	Now assume that the system of strong stationarity \eqref{eq:strong_stat}
	is satisfied.
	Let $h \in L^2(\Omega)$ with $v := S'(\bar u; h) \le 0$ on $\Omega_b$
	be given.
	We set
	$\xi_h := h - \AA v \in \KK(\bar u)\polar \subset H^{-1}(\Omega)$.
	Note that, in general, $v \not\in Z$,
	therefore we cannot use $v$ directly as a test function in the adjoint PDE.

	Due to the differentiability result \cref{thm:direc_diff},
	we have
	$\hat\xi := h - \AA v \in \KK(\bar u)\polar$.
	This implies $\varphi \hat\xi = \hat\xi$.
	In order to test the adjoint PDE,
	we approximate $\hat\xi \in H^{-1}(\Omega)$
	by a sequence $\seq{\xi_k}_{k \in \N} \subset L^2(\Omega)$
	such that $\xi_k \to \hat\xi$ in $H^{-1}(\Omega)$.
	This implies $\varphi \xi_k \to \varphi \hat\xi = \hat\xi$ in $H^{-1}(\Omega)$.
	Now,
	we can test the adjoint PDE by $z_k := \AA^{-1}(h - \varphi \xi_k) \in Z$
	and obtain
	\begin{equation*}
		\int_\Omega p \, (h - \varphi \xi_k) \, \d x
		+
		\dual{J_y(\bar y,\bar u) + \mu}{ z_k }
		+
		\int_{\Omega} z_k \, \d\nu
		.
	\end{equation*}
	Now, we have
	$z_k \to v$ in $H_0^1(\Omega)$
	and
	$z_k \to v$ in $C(\Omega_b)$
	due to \cref{thm:existenceh1,lem:hinfboundu}.
	Since the measure $\nu$ is supported on $\Omega_b$,
	we can pass to the limit $k \to \infty$
	and obtain
	\begin{equation*}
		\int_\Omega p h \, \d x
		-
		\dual{\hat\xi}{\varphi p}_{H_0^1(\Omega)}
		+
		\dual{J_y(\bar y,\bar u) + \mu}{ v }_{H_0^1(\Omega)}
		+
		\int_{\Omega} v \, \d\nu
		=
		0
		.
	\end{equation*}
	Now, we can use the sign conditions
	from the adjoint system and from $v$ and $\hat\xi$.
	This results in
	\begin{equation*}
		\int_\Omega p h \, \d x
		+
		\dual{J_y(\bar y,\bar u)}{ v }_{H_0^1(\Omega)}
		\ge
		0
		.
	\end{equation*}
	Since $h$ was arbitrary (as above) and $p = J_u(\bar y,\bar u)$,
	this shows \eqref{eq:B_stat_second}.
\end{proof}
Note that the second part of the proof also works in case $\Uad \ne L^2(\Omega)$.

\begin{remark}\leavevmode
	\label{rem:strongstat}
	\begin{enumerate}
		\item
			In the case that $\bar u$ is even a locally optimal control,
			one can use the results of \cref{sec:regularization}
			to skip some parts of the proofs,
			since the system of C-stationarity
			already includes the regularity of $p$
			and the adjoint equation.
			However, one still needs to extend the B-stationarity condition
			via density to \eqref{eq:B_stat_W1p} and \eqref{eq:B_stat_H1}
			to show the signs of $p$ and $\mu$.
		\item
			In the case $\Uad = L^2(\Omega)$,
			one can obtain the uniqueness of the multipliers.
			First, we infer $\lambda = 0$
			and, thus, $p$ is unique via \eqref{eq:strong_stat_2}.
			Let us argue that $\mu$ and $\nu$ are unique
			by using \eqref{eq:strong_stat_1}
			and the fact that the supports of $\mu$ and $\nu$
			are disjoint.
			To this end, let $\varphi \in C_c^\infty(\Omega)$
			be given such that
			$\varphi = 0$ on $\set{\bar y = y_a}$.
			Thus, $\pm \varphi \in \KK(\bar u)$
			and \eqref{eq:strong_stat_4}
			implies $\dual{\mu}{\varphi} = 0$.
			Hence,
			\eqref{eq:strong_stat_1}
			implies
			\begin{equation*}
				\dual{\nu}{\varphi}
				=
				-\dual{\AA\adjoint p + J_y(\bar y,\bar u)}{\varphi}
				\qquad\forall\varphi \in C_c^\infty, \varphi = 0 \text{ on } \set{\bar y = y_a}
				.
			\end{equation*}
			Since the support of $\nu$ is contained in $\set{\bar y = y_b}$,
			the measure $\nu$ is uniquely determined
			by the values of $\dual{\nu}{\varphi}$
			for these test functions $\varphi$.
			Since $p$ is unique, the uniqueness of $\nu$ follows.
			Consequently, the uniqueness of $\mu$ follows from \eqref{eq:strong_stat_1}.
	\end{enumerate}
\end{remark}

The next result addresses
the question of characterizing the normal cone to $\Ustate$,
which was left open in \cref{sec:primal_systems}.
\begin{lemma}
	\label{lem:normal_cone_via_strong}
	Let $\bar u, \tau \in L^2(\Omega)$ be given
	and set $\bar y := S(\bar u)$.
	Then, $\tau \in \NN_{\Ustate}(\bar u)$
	is equivalent to the existence of
	$p \in W_0^{1,q}(\Omega)$, $\mu \in H^{-1}(\Omega)$, $\nu \in \MM(\Omega)^+$
	such that
	$p \in H^1(\hat\Omega_a)$ for some open $\hat\Omega_a \supset \set{\bar y = y_a}$,
	$p = -\tau$,
	\begin{align*}
		\AA^\star p + \nu + \mu &= 0
	\end{align*}
	and the sign conditions
	\eqref{eq:strong_stat_3}, \eqref{eq:strong_stat_4}, \eqref{eq:strong_stat_5}
	are satisfied.
\end{lemma}
\begin{proof}
	Let $\tau \in L^2(\Omega)$ be arbitrary.
	We consider the auxiliary problem
	\begin{align*}
		\text{Minimize }\quad &{-\innerprod{\tau}{u}}
		\\
		\text{such that}\quad & u \in \Ustate
		.
	\end{align*}
	Note that this is a special case of problem \eqref{eq:P:unreg}
	with $\Uad = L^2(\Omega)$
	and $J(y,u) := -\innerprod{\tau}{u}$.
	Now,
	$\tau \in \NN_{\Ustate}(\bar u)$
	is equivalent to
	\begin{equation*}
		-\innerprod{\tau}{ h} \ge 0
		\qquad\forall h \in \TT_{\Ustate}(\bar u).
	\end{equation*}
	Using the characterization \eqref{eq:tangent_Ustate}
	of the tangent cone,
	this is
	equivalent to the B-stationarity \eqref{eq:B_stat_second} of $\bar u$
	for the auxiliary problem.
	Now, the assertion follows
	from the equivalency of B-stationarity and strong stationarity
	in \cref{thm:strong}.
\end{proof}

Using the characterization of the normal cone to $\Ustate$,
we can use 
directly the optimality system from
\cite[Theorem~1.1]{Wachsmuth2014:2},
i.e., with $\Ueff$ as control constraints.
Let us check that this does not yield a system of C-stationarity
in the case $\Uad \subset L^2(\Omega)$.
From the referenced optimality system,
we get the existence of
$\mu_1 \in H^{-1}(\Omega)$, $\hat\lambda \in L^2(\Omega)$ and $p_1 \in H_0^1(\Omega)$
such that the system
\begin{align*}
	\AA^\star p_1 + J_y(\bar y,\bar u) + \mu_1 &= 0 \quad\text{in } H^{-1}(\Omega), \\
	J_u(y,u) + \hat\lambda - p_1 &= 0 \quad\text{in } L^2(\Omega), \\
	p_1 & = 0 \quad\text{q.e.\ on } \qsupp(\bar\xi), \\
	\dual{\mu_1}{v}_{H^{-1}, H_0^1}
	&= 0 \quad\forall v \in H_0^1(\Omega), v = 0 \text{ q.e.\ on } \Omega_a,
	\\
	\dual{\mu_1}{\Phi p_1} &\ge 0
	\quad
	\forall \Phi \in W^{1,\infty}(\Omega)^+,
	\\
	\hat\lambda &\in \NN_{\Ueff}(u)
\end{align*}
is satisfied.
Next we use
\cref{thm:cones_Ueff,lem:normal_cone_via_strong}
to evaluate
$\hat\lambda \in \NN_{\Ueff}(\bar u)$.
This yields the existence of
$\lambda \in \NN_{\Uad}(\bar u)$,
$p_2 \in W_0^{1,q}(\Omega)$, $\mu_2 \in H^{-1}(\Omega)$, $\nu \in \MM(\Omega)^+$
such that
$\hat\lambda = \lambda - p_2$,
$p_2 \in H^1(\hat\Omega_a)$ for some open $\hat\Omega_a \supset \set{\bar y = y_a}$,
\begin{align*}
	\AA^\star p_2 + \nu + \mu_2 &= 0
\end{align*}
and the sign conditions
\eqref{eq:strong_stat_3}, \eqref{eq:strong_stat_4}, \eqref{eq:strong_stat_5}
are satisfied by $p_2$, $\mu_2$ and $\nu$, respectively.

By defining
$p = p_1 + p_2$ and
$\mu = \mu_1 + \mu_2$,
we arrive at
\eqref{eq:c_stat_1}
and
\eqref{eq:c_stat_2}.
The conditions \eqref{eq:c_stat_5} and \eqref{eq:c_stat_6}
on $\nu$ and $\lambda$ follow.
Moreover,
it is straightforward to see that
\eqref{eq:c_stat_3}, \eqref{eq:c_stat_4}
are satisfied.
However,
the sign condition
\eqref{eq:c_stat_45}
will, in general,
not be valid.

\printbibliography

\end{document}